\documentclass[11pt]{amsart}

\usepackage[colorlinks=true,linkcolor=blue]{hyperref}
\usepackage{amsmath}
\usepackage{amssymb}
\usepackage{amsthm}
\usepackage{bm} 

\usepackage{graphicx}
\usepackage{epstopdf}
\usepackage{epsfig}

\usepackage{tikz}
\usepackage{tikz-cd}

\usepackage{enumerate}

\theoremstyle{plain} 
\newtheorem{thm}{Theorem}[section]
\newtheorem{prop}[thm]{Proposition}
\newtheorem{lemma}[thm]{Lemma}

\theoremstyle{remark}
\newtheorem{remark}[thm]{Remark}

\theoremstyle{definition}
\newtheorem{defin}[thm]{Definition}

\newcommand{\QQ}{\mathbb{Q}}

\newcommand{\ZZ}{\mathbb{Z}}

\newcommand{\Qbar}{\overline{\QQ}}

\newcommand{\Kbar}{\overline{K}}
\newcommand{\kbar}{\overline{k}}

\DeclareMathOperator{\sgn}{sgn}
\DeclareMathOperator{\Par}{Par}

\DeclareMathOperator{\Orb}{Orb}
\DeclareMathOperator{\ch}{char}
\DeclareMathOperator{\Gal}{Gal}

\DeclareMathOperator{\res}{Res}

\DeclareMathOperator{\Aut}{Aut}

\newcommand{\dsps}{\displaystyle}

\newcommand{\mapa}{\bm{\alpha}}
\newcommand{\mapb}{\bm{\beta}}

\newcommand{\mapt}{\bm{\tau}}

\newcommand{\Pink}{\textup{Pink}}
\newcommand{\geom}{\textup{geom}}
\newcommand{\arith}{\textup{arith}}

\begin{document}

\title[Arboreal Galois groups of PCF quadratic polynomials]
{Arboreal Galois groups of postcritically finite quadratic polynomials: The strictly preperiodic case}

\author[R. L. Benedetto]{Robert L. Benedetto}
\address{Robert L. Benedetto \\ Department of Mathematics and Statistics \\ Amherst College \\ Amherst, MA 01002 \\ USA}
\email{rlbenedetto@amherst.edu}

\author[D. Ghioca]{Dragos Ghioca}
\address{Dragos Ghioca \\ Department of Mathematics \\ University of British Columbia \\ Vancouver, BC V6T 1Z2 \\ Canada}
\email{dghioca@math.ubc.ca}

\author[J. Juul]{Jamie Juul}
\address{Jamie Juul \\ Department of Mathematics \\ Colorado State University \\ Fort Collins, CO 80523 \\ USA}
\email{jamie.juul@colostate.edu}

\author[T. J. Tucker]{Thomas J. Tucker}
\address{Thomas J. Tucker\\Department of Mathematics\\ University of Rochester\\
Rochester, NY, 14620 \\ USA}
\email{thomas.tucker@rochester.edu}

%

\date{July 2, 2025}

\begin{abstract} 
In \cite{BGJT2}, we provided an explicit description of the arboreal Galois group of the postcritically finite polynomial $f(z) = z^2 +c$ in the special case when the critical point $0$ is periodic under the action of $f(z)$. In the current paper, we complete the picture for all postcritically finite polynomials by addressing the cases when $0$ is strictly preperiodic for the polynomial $f(z)$. 
\end{abstract}

\maketitle


\section{Introduction}


\subsection{General notation}

Let $K$ be a field of characteristic not equal to $2$
with algebraic closure $\Kbar$,
and let $f(z)\in K[z]$ be a polynomial of degree $2$.
After conjugating by a $K$-rational change of coordinates,
we may assume that $f(z)=z^2+c$ for some $c\in K$.

We consider the iterates $f^n$ of $f$ under composition,
where $f^0(z):=z$, and where $f^{n+1}=f\circ f^n$ for each integer $n\geq 0$.
A point $y\in\Kbar$ is said to be  \emph{preperiodic} if $f^n(y)=f^m(y)$ for some $n>m\geq 0$.
In the special case that $y$ is \emph{periodic} (i.e., $m=0$ with the above notation), then its \emph{exact period} is the smallest $n\geq 1$ for which $f^n(y)=y$.
An even further special case of periodicity occurs when $n=1$, in which case we say
$y$ is \emph{fixed} under the action of $f$.
If $y$ is preperiodic but not periodic, then we say it is \emph{strictly preperiodic}.

Given a point $x_0\in K$, then for every integer $n\geq 0$, we define
\[ K_n:=K_{x_0,n}:= K(f^{-n}(x_0))
\qquad\text{and}\qquad
G_n:= G_{x_0,n}:=\Gal(K_n/K) \]
to be the $n$-th preimage field and its associated Galois group.
Note that $\cdots K_3/K_2/K_1/K$ is a tower of field extensions,
which we view as contained in $\Kbar$.
Thus, we may further define
\[ K_\infty:=K_{x_0,\infty}:=\bigcup_{n\geq 0} K_{x_0,n}
\qquad\text{and}\qquad
G_\infty:=G_{x_0,\infty}:=\Gal(K_{x_0,\infty}/K)\cong \varprojlim_n G_{x_0,n} .\]

If the backward orbit
\[ \Orb_f^-(x_0) := \bigcup_{n\geq 0} f^{-n}(x_0) \]
contains no critical values of $f$,
then each $f^{-n}(x_0)$ has exactly $2^n$ elements. If, in addition,
$x_0$ is not periodic under $f$, then the sets $f^{-n}(x_0)$ are pairwise disjoint,
and hence $\Orb_f^-(x_0)$ has the structure of an infinite binary rooted tree $T_\infty$,
with $x\in f^{-(n+1)}(x_0)$ connected to $f(x)\in f^{-n}(x_0)$ by an edge.
Thus, the action of the Galois group $G_{\infty}$ on the backward orbit induces
an embedding of $G_{\infty}$ into the automorphism group $\Aut(T_\infty)$ of the tree.
Similarly, for each $n\geq 0$, the action of $G_n$ on $f^{-n}(x_0)$ induces an embedding
of $G_n$ into the automorphism group $\Aut(T_n)$ of the finite binary rooted tree $T_n$
with $n$ levels. For this reason, the groups $G_n$ and $G_\infty$ have come to be
known as \emph{arboreal} Galois groups.
Moreover, given our interest in this action, whenever we discuss homomorphisms
or isomorphisms between groups acting on trees, we always mean homomorphisms
that are equivariant with respect to this action. The problem of fully understanding the arboreal Galois groups has generated a great deal of research in the recent years; see 
\cite{ABCCF, BDGHT, BFHJY, BGJT1, BGJT2, BJ19, JKMT, Juul, PinkPCF}, for example.


\subsection{Postcritically finite quadratic polynomials}

In this paper, we consider the case that $f$ is \emph{postcritically finite}, or PCF,
meaning that all of the the critical points of $f$ are preperiodic.
Since we have assumed $f(z)=z^2+c$, the critical points are $0$ and $\infty$,
with $\infty$ necessarily fixed; thus, to say that $f$ is PCF is equivalent to saying that
$0$ is preperiodic under $f$. In this case, it is well known that $G_{\infty}$ is
of infinite index in $\Aut(T_{\infty})$.
Following the notation in \cite{PinkPCF}, there is a maximal integer $r\geq 1$ such that
the values $f(0), f^2(0),\ldots,f^r(0)$ are all distinct.
That is, we have $f^{r+1}(0)=f^{s+1}(0)$ for some minimal integers $r>s\geq 0$.
Equivalently, since the two preimages of $f(y)$ are $\pm y$, we have $f^r(0)=-f^s(0)$
for minimal integers $r>s\geq 0$.
We are interested in the case that the critical point $0$ is strictly preperiodic,
in which case $s\geq 1$, and the (\emph{strict}) \emph{forward orbit}
\[ \Orb_f^+(0) := \{f^i(0) : i\geq 1 \}\]
of $0$ under $f$ consists of precisely $r$ points.
Also, we note that the point $f^{s+1}(0)=f^{r+1}(0)$ is periodic of exact period $r-s\geq 1$,
preceded by a tail $\{0,f(0),\ldots,f^s(0)\}$ of cardinality $s+1\geq 2$.


\subsection{Pink's work over function fields}
\label{ssec:PinkSummary}

In \cite{PinkPCF}, Pink describes the group $G_{\infty}$ for each of the various
choices of $r,s$ when the quadratic polynomial $f$ is PCF, in the case that $K=\kbar(t)$
is a rational function field over an algebraically closed field $\kbar$,
and that the root point of the preimage tree is $x_0=t$.
Pink denotes this group $G^{\geom}$, 
and he proves that it is isomorphic to a subgroup of $\Aut(T_{\infty})$ that he simply calls $G$,
but which we denote $G^{\Pink}_{r,s,\infty}$.
(When $s=0$, we sometimes write simply $G^{\Pink}_{r,\infty}$.)
He defines $G^{\Pink}_{r,s,\infty}$ via explicit (topological) generators,
each arising from the action of inertia in the context of $G^{\geom}$.

When $K=k(t)$ for $k$ \emph{not} algebraically closed, Pink denotes
the resulting group $G_{\infty}$ as $G^{\arith}$, and he describes how it fits
into a short exact sequence
\[ 1 \longrightarrow G^{\Pink}_{r,s,\infty} \longrightarrow G^{\arith}
\longrightarrow \Gal(\kbar/k)/N \longrightarrow 1, \]
for some normal subgroup $N$ of $\Gal(\kbar/k)$ depending on $r$, $s$, and $k$.


\subsection{Our results}

For each pair of integers $r>s\geq 1$, we construct
subgroups $B_{r,s,\infty}\subseteq M_{r,s,\infty}$ of
$\Aut(T_{\infty})$ (see Definition~\ref{def:MRSiroot} and also Theorem~\ref{thm:MRSiroot}), coinciding with Pink's group
$G^{\Pink}_{r,s,\infty}\cong G^{\geom}\subseteq G^{\arith}$, and we
show that the arboreal Galois group $G_{\infty}$ is isomorphic to a
subgroup of $M_{r,s,\infty}$. In this paper, we work under the assumption that  $s\ge 1$,
since we covered the special case $s=0$ (i.e., $0$ is periodic under the action of $f$)
in our earlier paper \cite{BGJT2}. 

We also exclude the case of the Chebyshev quadratic polynomial $f(z)=z^2-2$ (which is the case
$s=1$ and $r=2$),
since the Galois group in this case has an alternative classical description.
Indeed, the quadratic Chebyshev polynomial satisfies
\[ f^n(z+1/z) = z^{2^n} + 1/z^{2^n} \quad\text{for all}\quad n\geq 0. \]
Thus, arboreal extensions arising from the Chebyshev polynomial $f(z)=z^2-2$,
and hence their associated Galois groups, can be expressed in terms of
classical Kummer extensions.  

Our arguments apply over general fields
with arbitrary base points, rather than restricting to the case
$K=k(t)$ with base point $t$.  Our approach is also
more concrete than Pink's; we define the groups $B_{r,s,\infty}$ and
$M_{r,s,\infty}$ not by generators but rather as the set
of all $\sigma\in\Aut(T_{\infty})$ satisfying certain parity conditions,
which also describe the action of elements of $G_\infty$ on $\sqrt{2}$ and on $2$-power roots of unity in $K_\infty$.
As a by-product of our method, we obtain necessary and sufficient conditions for $G_{\infty}$
to be the whole group $M_{r,s,\infty}$ (see Theorem~\ref{thm:condMisGarith}). 

Our approach follows the general strategy of our previous paper \cite{BGJT2},
but there are additional technical complications arising in the strictly preperiodic cases.
In particular, one of the cases ($s=2$ and $r=3$) requires a significantly more delicate argument
than any other case. (Analogous complications arose in Pink's analysis of the same case in \cite{PinkPCF}.)



\subsection{Tree labelings}
\label{ssec:label}

To prove and even state our main results, we 
label the nodes of the binary rooted trees $T_n$ and $T_{\infty}$,
using words in the two symbols $a$, $b$.
That is, for each integer $m\geq 0$ and each node $y$ at the $m$-th level of the tree,
we assign $y$ a \emph{label} in the form of a word $w\in\{a,b\}^m$ of length $m$,
in such a way that for every such $m$ and $y$, the two nodes lying above $y$
have labels $wa,wb\in\{a,b\}^{m+1}$. (Of course, in the tree $T_n$, this latter restriction
is vacuous for nodes $y$ in the top level $m=n$.)
See Figure~\ref{fig:treelabel} for an example of a labeling of the tree $T_3$.
Although the root node has the empty label $()$, we will often denote it as $x_0$.

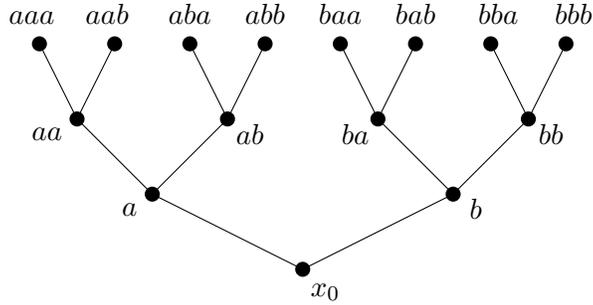
\begin{figure}
\begin{tikzpicture}
\path[draw] (.4,3.5) -- (.9,2.5) -- (1.4,3.5);
\path[draw] (2.4,3.5) -- (2.9,2.5) -- (3.4,3.5);
\path[draw] (4.4,3.5) -- (4.9,2.5) -- (5.4,3.5);
\path[draw] (6.4,3.5) -- (6.9,2.5) -- (7.4,3.5);
\path[draw] (.9,2.5) -- (1.9,1.5) -- (2.9,2.5);
\path[draw] (4.9,2.5) -- (5.9,1.5) -- (6.9,2.5);
\path[draw] (1.9,1.5) -- (3.9,0.5) -- (5.9,1.5);
\path[fill] (.4,3.5) circle (0.1);
\path[fill] (1.4,3.5) circle (0.1);
\path[fill] (2.4,3.5) circle (0.1);
\path[fill] (3.4,3.5) circle (0.1);
\path[fill] (4.4,3.5) circle (0.1);
\path[fill] (5.4,3.5) circle (0.1);
\path[fill] (6.4,3.5) circle (0.1);
\path[fill] (7.4,3.5) circle (0.1);
\path[fill] (0.9,2.5) circle (0.1);
\path[fill] (2.9,2.5) circle (0.1);
\path[fill] (4.9,2.5) circle (0.1);
\path[fill] (6.9,2.5) circle (0.1);
\path[fill] (1.9,1.5) circle (0.1);
\path[fill] (5.9,1.5) circle (0.1);
\path[fill] (3.9,0.5) circle (0.1);
\node (a) at (1.6,1.3) {$a$};
\node (b) at (6.2,1.3) {$b$};
\node (aa) at (0.5,2.3) {$aa$};
\node (ab) at (3.2,2.3) {$ab$};
\node (ba) at (4.6,2.3) {$ba$};
\node (bb) at (7.2,2.3) {$bb$};
\node (aaa) at (0.3,3.85) {$aaa$};
\node (aab) at (1.3,3.9) {$aab$};
\node (aba) at (2.4,3.9) {$aba$};
\node (abb) at (3.4,3.9) {$abb$};
\node (baa) at (4.4,3.9) {$baa$};
\node (bab) at (5.4,3.9) {$bab$};
\node (bba) at (6.5,3.9) {$bba$};
\node (bbb) at (7.5,3.9) {$bbb$};
\node (x0) at (4.2,0.2) {$x_0$};
\end{tikzpicture}
\caption{A labeling of $T_3$}
\label{fig:treelabel}
\end{figure}

We usually consider the nodes of $T_\infty$ as corresponding to the backward orbit
$\Orb_f^-(x_0)\in\Kbar$ of $x_0\in K$ under $f(z)=z^2+c\in K[z]$. Thus, we will often
conflate a point $y\in f^{-n}(x_0)$ with the corresponding node $y$ of the tree.
Having assigned a labeling to the tree, we will also sometimes conflate a node $y$
with its label.
On the other hand, when further clarity is needed for the backward orbit $\Orb_f^-(x_0)\in\Kbar$,
viewed as a tree of preimages, we will often write
the value $y\in f^{-n}(x_0)\subseteq\Kbar$ corresponding to the node
with label $w\in\{a,b\}^n$ as $y=[w]$.

Having labeled the tree, any tree automorphism
$\sigma\in\Aut(T_\infty)$ or $\sigma\in\Aut(T_n)$
must satisfy the following conditions.
\begin{enumerate}
\item For every level $m\geq 0$ (up to $m\leq n$ for $T_n$),
$\sigma$ permutes the labels in $\{a,b\}^m$, and
\item For every level $m\geq 0$ (up to $m\leq n-1$ for $T_n$),
for each word $s_1\ldots s_m\in \{a,b\}^m$, we have either
\[ \sigma(s_1 \cdots s_m a) = \sigma(s_1 \cdots s_m) a
\quad\text{and}\quad \sigma(s_1 \cdots s_m b) = \sigma(s_1 \cdots s_m) b \]
or
\[ \sigma(s_1 \cdots s_m a) = \sigma(s_1 \cdots s_m )b
\quad\text{and}\quad \sigma(s_1 \cdots s_m) b = \sigma(s_1 \cdots s_m) a . \]
\end{enumerate}
For any tree automorphism $\sigma$ and $m$-tuple $x\in\{a,b\}^m$, we define 
the \emph{parity} $\Par(\sigma,x)$ of $\sigma$ at $x$ to be
\[ \Par(\sigma,x) := \begin{cases}
0 & \text{ if } \sigma(xa)=\sigma(x)a \text{ and } \sigma(xb)=\sigma(x)b
\\
1 & \text{ if } \sigma(xa)=\sigma(x)b \text{ and } \sigma(xb)=\sigma(x)a
\end{cases}. \]
Thus, any set of choices of $\Par(\sigma,x)$ for each node
$x$ of $T_{\infty}$ (respectively, $T_{n-1}$) determines a
unique automorphism $\sigma\in \Aut(T_{\infty})$ (respectively, $\sigma\in\Aut(T_n)$).

Note that if $\sigma(x)=x$, then $\Par(\sigma,x)$ is $0$ if $\sigma$ fixes
the two nodes above $x$, or $1$ if it transposes them. However, $\Par(\sigma,x)$
is defined even if $\sigma(x)\neq x$, although in that case its value also depends on the
labeling of the tree.

We further define $\sgn(\sigma,x)=(-1)^{\Par(\sigma,x)}$.
We have the following elementary relations:
\begin{equation}
\label{eq:sgn1}
\sgn(\sigma \tau, x) = \sgn\big(\sigma, \tau(x) \big) \cdot \sgn(\tau, x) ,
\end{equation}
and hence
\begin{equation}
\label{eq:sgn2}
\Par(\sigma\tau,x) = \Par\big(\sigma, \tau(x)\big) + \sgn\big(\sigma,\tau(x)\big) \Par(\tau,x) .
\end{equation}
Equation~\eqref{eq:sgn2} follows from equation~\eqref{eq:sgn1} by writing
$\Par(\cdot,\cdot)=(1-\sgn(\cdot,\cdot))/2$, or simply by checking the four
possible choices of $\Par(\tau,x)$ and $\Par(\sigma,\tau(x))$.

In particular, working modulo $2$, we have
\begin{equation}
	\label{eq:sgn22}
	\Par(\sigma\tau,x) \equiv \Par\big(\sigma, \tau(x)\big) + \Par(\tau,x)  \pmod{2} .
\end{equation}

\begin{defin}
\label{def:MRSiroot}
Fix a labeling of $T_{\infty}$, and fix integers $r>s\geq 1$.
For any word $x$ in the symbols $\{a,b\}$ and any $\sigma\in\Aut(T_{\infty})$,
define
\[ P^a_{r,s}(\sigma,x) := \sum_{w\in\{a,b\}^{r-1}} \Par(\sigma,xaw)
+ \sum_{w'\in\{a,b\}^{s-1}} \Par(\sigma,xbw') \in\ZZ/2\ZZ \]
and
\[ P^b_{r,s}(\sigma,x) := \sum_{w\in\{a,b\}^{r-1}} \Par(\sigma,xbw)
+ \sum_{w'\in\{a,b\}^{s-1}} \Par(\sigma,xaw')  \in\ZZ/2\ZZ .\]
Define $M_{r,s,\infty}$ to be the set of all $\sigma\in\Aut(T_{\infty})$ for which
\begin{equation}
\label{eq:PRSdef}
P^a_{r,s}(\sigma,x_1) = P^b_{r,s}(\sigma,x_1) =
P^a_{r,s}(\sigma,x_2) = P^b_{r,s}(\sigma,x_2) \in\ZZ/2\ZZ
\end{equation}
for all nodes $x_1,x_2$ of $T_{\infty}$.
For $\sigma\in M_{r,s,\infty}$, define $P_{r,s}(\sigma)\in\ZZ/2\ZZ$ to be this common value.
Define $B_{r,s,\infty} := \{\sigma\in M_{r,s,\infty} : P_{r,s}(\sigma)=0 \}$.
\end{defin}

For  $0\leq m\leq n\leq \infty$, it will be convenient to define homomorphisms
\[ \res_{n,m}:\Aut(T_n)\rightarrow \Aut(T_m) \]
given by restricting elements of $\Aut(T_n)$ to the $m$-th level of the tree.
In particular, for each integer $n\geq 1$, we may
define $B_{r,s,n}:=\res_{\infty,n}(B_{r,s,\infty})$ and $M_{r,s,n}:=\res_{\infty,n}(M_{r,s,\infty})$.


\subsection{Main result}
\label{ssec:mainthm}
We can now state our main theorem;
for a more general version, see Theorem~\ref{thm:mainthm}.

\begin{thm}
\label{thm:condition}
Let $K$ be a field of characteristic not equal to $2$,
and let $f(z)=z^2+c\in K[z]$ such that $0$ is strictly preperiodic under $f$.
Let $r>s\geq 1$ be  minimal so that $f^r(0)=-f^s(0)$, and assume $r\ge 4$ and $s\ge 2$.  

Let $x_0\in K$, and define $K_{x_0,n}=K(f^{-n}(0))$,
$K_{x_0,\infty}=\bigcup_{n=1}^\infty K_{x_0,n}$,
$G_{x_0,n}=\Gal(K_{x_0,n}/K)$, and $G_{x_0,\infty}=\Gal(K_{x_0,\infty}/K)$.
In addition, define $D_1,\dots, D_r \in K$ by 
\[ D_i:=\begin{cases}
	x_0-c & \text{ if } i=1,\\
	f^i(0)-x_0 & \text{ if } i\geq 2.
\end{cases} \] 
Then the following are equivalent.
\begin{enumerate}
	\item $[k(\zeta_4, \sqrt{D_1}, \dots, \sqrt{D_r}):k]= 2^{r+1}.$
	\item $[K_{x_0,r+1}:K]=|M_{r,s,r+1}|$.
	\item $G_{x_0,r+1}=M_{r,s,r+1}$.
	\item $G_{x_0,n}=M_{r,s,n}$ for all $n\geq 1$.
	\item $G_{x_0,\infty}\cong M_{r,s,\infty}$.
\end{enumerate}
\end{thm}

In Theorem~\ref{thm:mainthm} we address the remaining cases: $s=1$ and $r\ge 3$, and also $(s,r)=(2,3)$.
There are technical subtleties appearing in these other cases, especially when $(s,r)=(2,3)$.


\subsection{Outline of the paper}
In Section~\ref{sec:preliminary} we present
several auxiliary results that we will use in later sections.
We then divide our polynomials $z^2+c$ with strictly preperiodic critical point into three cases:
the \emph{long-tail case}, where $r>s\ge 2$,
the \emph{special long-tail case}, where $(r,s) = (3,2)$,
and the \emph{short-tail, non-Chebyshev case}, where $r\ge 3$ and $s=1$.

In Section~\ref{sec:zeta4longtail}, we treat the long-tail case.
We begin by writing $\sqrt{-1}$ as an explicit combination of
inverse images of a point in Lemma \ref{lem:iroot}.
Then, in Theorem~\ref{thm:PRSembed}, we prove that the image of our iterated Galois group
$G_\infty$ in $\Aut(T_\infty)$ must be contained in $M_{r,s,\infty}$,
by describing its action on $\sqrt{-1}$ in terms of  $P^a_{r,s}$ and  $P^b_{r,s}$.

In Section~\ref{sec:Special}, we deal with the special long-tail case
$(r,s)=(3,2)$, which requires additional attention.  In Lemma~\ref{lem:root2special},
we show that it is also possible to write
$\sqrt{2}$ as an explicit combination of inverse images of a point.
(Note that, for a general root point $x_0$,
it is not possible to do this for $r>s\ge 2$ when
$(r,s) \not= (3,2)$.)  We further introduce a new function $R_{3,2}$ in
Definition~\ref{def:R2root} and use it to define a special subgroup
$\tilde{M}_{3,2,\infty}$ of $M_{3,2,\infty}$. Then, in Theorem~\ref{thm:R32embed},
we show that the image of
$G_\infty$ in $\Aut(T_\infty)$ must be contained in
$\tilde{M}_{3,2,\infty}$, by describing its action on $\sqrt{2}$ in terms of $R_{3,2}$.

In Section~\ref{sec:shorttail}, we treat the short-tail,
non-Chebyshev case.  Following the pattern of the previous two
sections, we write $\sqrt{2}$ as an explicit combination of inverse
images of a point (in Lemma~\ref{lem:root2}), define a subgroup
$\tilde{M}_{r,1,\infty}$ of $B_{r,1,\infty}$ in terms of a function
$R_{r,1}$ (in Definition~\ref{def:MRroot2}), and then show that the
image of our iterated Galois group $G_\infty$ in $\Aut(T_\infty)$ must
be contained in $\tilde{M}_{r,1,\infty}$ (in Theorem~\ref{thm:PR1embed}).

Next, we relate our groups to Pink's in
Section~\ref{sec:PinkGroupsGeom}, by first showing that Pink's
generators for $G_{r,s,\infty}^{\Pink}$ lie in $B_{r,s,\infty}$ in the
long-tail case, and in corresponding groups ${\tilde B}_{r,s,\infty}$
when $s=1$ or when $(r,s) = (3,2)$. We then 
prove that Pink's group $G_{r,s,\infty}^{\Pink}$ in fact
coincides with $B_{r,s,\infty}$ or ${\tilde B}_{r,s,\infty}$
by comparing their rates of growth.

Finally, we derive explicit conditions for the arboreal Galois groups
to be of maximal size in Section~\ref{sec:obtain}.
The full version of our main result Theorem~\ref{thm:condition}
appears here as Theorem \ref{thm:mainthm}.


\section{Preliminary results}
\label{sec:preliminary}



\subsection{Iterated preimages of given point} 

We proved the following elementary result in \cite[Proposition~2.1]{BGJT2}.
We will use it extensively in the current paper, so we restate it here for the convenience of the reader.

\begin{prop}
\label{prop:key}
Let $K$ be a field of characteristic not equal to $2$.
Let $c\in K$, define $f(z)=z^2+c$, let $y\in\overline{K}$, and let $m\geq 1$.
Choose $\alpha_1,\ldots,\alpha_{2^{(m-1)}}\in f^{-m}(y)$ so that
the roots of $f^m(z)=y$, repeated according to multiplicity, are precisely
\begin{equation}
\label{eq:alpharoots}
\pm\alpha_1, \ldots, \pm \alpha_{2^{(m-1)}} .
\end{equation}
Then
\[ \big( \alpha_1 \alpha_2 \cdots \alpha_{2^{(m-1)}} \big)^2 =
\begin{cases}
f^m(0)-y & \text{ if } m\geq 2, \\
y-f(0) & \text{ if } m=1.
\end{cases} \]
\end{prop}

Our next proposition will be helpful for describing some subgroups of $\Aut(T_\infty)$.
When applying it, $X$ will be the set of nodes of $T_{\infty}$,
and $\Gamma$ will be a subgroup of $\Aut(T_\infty)$, which acts on $X$.
(Often, we will have $\Gamma=\Aut(T_\infty)$, in fact.) An appropriate function $P$,
as described in the proposition, then allows us to define a subgroup of $\Gamma$.

\begin{prop}
\label{prop:Pgroup}
Let $\Gamma$ be a group with identity element $e$,
and let $H$ be a group with subgroup $H_0$.
Let $X$ be a nonempty set, and let $P:\Gamma\times X \to H$ be a function.
Define
\[ \dsps G:=\{\sigma\in\Gamma : \forall x_1,x_2\in X, \, \, P(\sigma,x_1)=P(\sigma,x_2)\in H_0 \} .\]
Suppose that $e\in G$, and that for all $\sigma\in G$, $\tau\in\Gamma$, and $x\in X$, we have
\begin{equation}
\label{eq:Phomom}
P(\sigma) P(\tau,x) = P(\sigma\tau,x),
\end{equation}
where $P(\sigma):=P(\sigma,x)\in H_0$ is the constant value of $P(\sigma,\cdot)$.
Then $G$ is a subgroup of $\Gamma$, and $P:G\to H_0$ is a group homomorphism.
\end{prop}

\begin{proof}
By hypothesis, $G$ is nonempty. Given $\sigma,\tau\in G$ and $x_1,x_2\in X$, we have
\[ P(\sigma\tau, x_1) = P(\sigma) P(\tau,x_1) = P(\sigma) P(\tau,x_2) = P(\sigma\tau, x_2),\]
and this common value lies in $H_0$, since $P(\sigma), P(\tau,x_1) \in H_0$.
Thus, $\sigma\tau\in G$. In addition,
\[ P(\sigma) P(\sigma^{-1}, x_1) = P(\sigma\sigma^{-1}, x_1) = P(e, x_1)
= P(e, x_2)= P(\sigma\sigma^{-1}, x_2) = P(\sigma) P(\sigma^{-1}, x_2) .\]
Multiplying both sides on the left by $P(\sigma)^{-1}\in H_0$, we have
\[ P(\sigma^{-1}, x_1) = P(\sigma^{-1}, x_2) = P(\sigma)^{-1} P(e, x_1) \in H_0, \]
proving that $\sigma^{-1}\in G$.
Thus, $G$ is a subgroup of $\Gamma$; and by equation~\eqref{eq:Phomom},
the map $\sigma\mapsto P(\sigma)$ is a homomorphism.
\end{proof}


\subsection{Ramification in number fields generated by a PCF parameter}

The results in this section concern the fields of definition of maps $f(z)=z^2+c$ with particular
critical orbit structures. We will not use them until Section~\ref{ssec:PinkGroupsArith}.

Write $f_x(z)=z^2+x\in\ZZ[x,z]$.
For any integer $n\geq 1$, observe that $f_x^n(0)\in\ZZ[x]$ is a polynomial in $x$.
For each pair of integers $r>s\geq 0$, define $F_{r,s}(x) = f_x^r(0) + f_x^s(0)$.
The roots of $F_{r,s}\in\ZZ[x]$ are those values of $c\in\Qbar$ for which $f_c^r(0)=-f_c^s(0)$.

\begin{lemma}
\label{lem:mod2iter}
For any integers $m,n\geq 1$, we have
\[ f_x^n(0) \equiv \sum_{i=0}^{n-1} x^{2^i} \pmod{2\ZZ[x]} . \]
\end{lemma}

\begin{proof}
We proceed by induction on $n$.
For $n=1$, we have $f_x(0)=x$ as desired.
Assuming it holds for $n$, then modulo~2, we have
\[ f_x^{n+1}(0) = \big(f_x^n(0))^2 + x  \equiv 
x+\sum_{i=0}^{n-1} \Big(x^{2^i}\Big)^2 = \sum_{i=0}^{n} x^{2^i}. \qedhere \]
\end{proof}

\begin{lemma}
\label{lem:valFrs}
Let $r>s\geq 1$ be integers.
Let $c\in\Qbar$ be a root of the polynomial $F_{r,s}(x)=f_x^r(0)+f_x^s(0)$
such that the $r+1$ iterates $\{f_c^i(0) \, | \, 0\leq i\leq r\}$ are all distinct.
Let $v_2$ be the valuation on $\QQ_2(c)$ normalized so that $v_2(2)=1$, and let $n:=r-s$.
Then:
\begin{enumerate}
\item
For every $i\geq 1$, we have
$\dsps v_2(f_c^{i}(0))=\begin{cases}
2^{-s} & \text{ if } n\nmid s \text{ and } n|i, \\
2^{-s}-1 & \text{ if } n|s \text{ and } n|i, \\
0 & \text{ if } n\nmid i.
\end{cases}$
\item
If $n\nmid s$, then the ramification index $e(\QQ_2(c)/\QQ_2)$ is exactly $2^s$.
\item
If $n|s$, then the ramification index $e(\QQ_2(c)/\QQ_2)$ is exactly $2^s-1$.
\end{enumerate}
\end{lemma}

\begin{proof}
Since the $r+1$ listed iterates are distinct, $c$ must be a root of the relevant
Misiurewicz polynomial; see Section~1 of \cite{Gok20} or equation~(1) of \cite{BG23}.
Statement~(1) is therefore the content
of Theorem~1.3 of \cite{Gok20} with $i=n=r-s$.

Define $\dsps g(x):=x^{2^{n-1}} + x^{2^{n-2}} + \cdots + x$.
Then Lemma~\ref{lem:mod2iter} gives us
\begin{equation}
\label{eq:Ppower2s}
F_{r,s}(x) \equiv \sum_{i=s}^{r-1} x^{2^i}  \equiv \big( g(x) \big)^{2^s} \pmod{2}.
\end{equation}
Since $g'(x)\equiv 1 \pmod{2}$, we have that $g(x)$ is separable modulo~2.
Therefore, it follows from equation~\eqref{eq:Ppower2s}
that the desired ramification index satisfies $e(\QQ_2(c)/\QQ_2)\leq 2^s$.

If $n\nmid s$, then statement~(1) of this lemma says that
$v_2(f_c^{n}(0))=2^{-s}$, and hence $e(\QQ_2(c)/\QQ_2)$ is divisible by $2^s$.
Since we just showed that $e(\QQ_2(c)/\QQ_2)\leq 2^s$, we must have
$e(\QQ_2(c)/\QQ_2)=2^s$, proving statement~(2).

If $n|s$ and $s\geq 2$, then statement~(1) of this lemma shows that
$e(\QQ_2(c)/\QQ_2)$ is divisible by $2^s - 1$.
Combined with the above fact $e(\QQ_2(c)/\QQ_2)\leq 2^s$, we have
$e(\QQ_2(c)/\QQ_2)=2^s-1$, proving statement~(3) when $s\geq 2$.

Finally, if $n|s$ and $s=1$, then we must have $r=2$, so that
$F_{r,s}(x)=x^2+2x=x(x+2)$, and hence $c\in\{0,-2\}$.
(Actually, the case $c=0$ is excluded by the
assumption that the first $r+1$ iterates $f_c^i(0)$ are distinct.)
Thus, $\QQ_2(c)=\QQ_2$, so that 
$e(\QQ_2(c)/\QQ_2)=1 = 2^s-1$, completing the proof of statement~(3).
\end{proof}

For the case $s=1$, we have the following result.

\begin{prop}
\label{prop:root2notQc}
Let $r\geq 3$, and let $c\in\Qbar$ be a root of the polynomial $F_{r,1}(x)=f_x^r(0)+x$.
Then $\sqrt{2}\not\in\QQ(c)$.
\end{prop}

\begin{proof}
We will prove the stronger result that $\sqrt{2}\not\in\QQ_2(c)$.
Let $n:=r-1\geq 2$, and
let $a:=f_c^{n-1}(0)$ and $y:= f_c(a) = f_c^n(0)$, so that $a,y \in\QQ_2(c)$.
By Lemma~\ref{lem:valFrs}.(1), we have $v_2(a)=0$ and $v_2(y)=1/2$.

Because $y^2=f_c(y)-c=f_c^r(0)-c$ and $a^2=f_c(a)-c=y-c$, we have
\[ \big(y-a\sqrt{2}\big)\big(y+a\sqrt{2}\big) = y^2 - 2a^2 = f_c^r(0) - c - 2( y- c ) = F_{r,1}(c) - 2y = -2y ,\]
where the final equality is because $F_{r,1}(c)=0$.
Thus, if $\sqrt{2}\in\QQ_2(c)$, we would have
\begin{equation}
\label{eq:v2product}
v_2\big( y - a\sqrt{2} \big) + v_2\big( y + a\sqrt{2} \big) = 1 + v_2(y) = \frac{3}{2}.
\end{equation}
If $v_2( y - a\sqrt{2}) > v_2( y + a\sqrt{2} )$, then by the triangle equality for valuations, we would have
\[ v_2\big(y-a\sqrt{2}\big) > v_2\big(y+a\sqrt{2}\big) = v_2\big( 2a\sqrt{2} \big) = \frac{3}{2},\]
contradicting equation~\eqref{eq:v2product}.
Similarly, we cannot have $v_2( y - a\sqrt{2}) < v_2( y + a\sqrt{2} )$, either.

Thus, $v_2( y - a\sqrt{2}) = v_2( y + a\sqrt{2} )$, and hence equation~\eqref{eq:v2product}
yields $v_2( y - a\sqrt{2})=3/4$.
Since $y-a\sqrt{2}\in\QQ_2(c)$, it would follow that
the ramification index of $\QQ_2(c)/\QQ_2$ is at least $4$,
contradicting Lemma~\ref{lem:valFrs}.(2), which says that $e(\QQ_2(c)/\QQ_2)=2$.
Therefore, we must have $\sqrt{2}\not\in\QQ_2(c)$.
\end{proof}

For $s\geq 2$, we have the following result.

\begin{prop}
\label{prop:rootsnotQc}
Let $r > s \geq 2$, and let $c\in\Qbar$ be a root of the polynomial $F_{r,s}(x)=f_x^r(0)+f_x^s(0)$.
Then $\sqrt{2},\sqrt{-1},\sqrt{-2}\not\in\QQ(c)$.
\end{prop}

\begin{proof}
Let $n:=r-s$. If $n|s$, then by Lemma~\ref{lem:valFrs}.(3), the extension $\QQ_2(c)/\QQ_2$
has odd ramification degree. Therefore, $\QQ_2(c)$ cannot contain any of
$\sqrt{2},\sqrt{-1},\sqrt{-2}$, which each have ramification index~2 over $\QQ_2$.
For the remainder of the proof, then, we assume that $n\nmid s$.

Pick $\ell\geq 1$ such that $n|(s+\ell)$, let $\alpha:=f_c^\ell(0)$, and let $a:=\alpha^{2^{s-1}}$.
Since $n\nmid s$, we also have $n\nmid \ell$, and therefore
$v_2(a)=v_2(\alpha)=0$ by Lemma~\ref{lem:valFrs}.(1).

Let $\omega:=f_c^n(0)$ and $y:=\omega^{2^{s-1}}$, so that
by Lemma~\ref{lem:valFrs}.(1), we have $v_2(\omega)=2^{-s}$ and hence $v_2(y)=1/2$.

Also define $R:=f_c^{r-1}(0)$ and $S:=f_c^{s-1}(0)$. By Lemma~\ref{lem:mod2iter}, we have
\[ y + S \equiv \bigg( \sum_{i=0}^{n-1} c^{2^i} \bigg)^{2^{s-1}} + \sum_{i=0}^{s-2} c^{2^i}
\equiv \sum_{i=s-1}^{r-2} c^{2^i} + \sum_{i=0}^{s-2} c^{2^i} = \sum_{i=0}^{r-2} c^{2^i}
\equiv R \pmod{2\ZZ[c]} , \]
and hence
\[ y^2 + 2yS + S^2 = (y+S)^2 \equiv R^2 \pmod{4\ZZ[c]} .\]
On the other hand, we also have $R^2+S^2 + 2c = F_{r,s}(c)=0$, and therefore
\begin{equation}
\label{eq:ysimp}
y^2 = y^2 - F_{r,s}(c) \equiv \big( R^2 - S^2 - 2yS \big) - \big(R^2 + S^2 + 2c \big)
\equiv 2(c+ S^2 + yS) \pmod{4\ZZ[c]} .
\end{equation}

Let $\gamma$ be any of $\sqrt{2},\sqrt{-2}$, or $1+\sqrt{-1}$.
Let $\tilde{\gamma}:=-\gamma$ if $\gamma=\sqrt{\pm 2}$,
or $\tilde{\gamma}:=2-\gamma$ if $\gamma=1+\sqrt{-1}$. Define $Q(t)\in\ZZ[c][t]$ to be
\[  Q(t) := (t-a\gamma)(t-a\tilde{\gamma}) =
\begin{cases}
t^2 \mp 2a^2 & \text{ if } \gamma=\sqrt{\pm 2}, \\
t^2 - 2at+2a^2 & \text{ if } \gamma=1+\sqrt{-1}.
\end{cases} \]
Noting that $0 < 2v_2(\omega)\leq v_2(y) < 1$, then, and working modulo $2\omega^2\ZZ[c]$, we have
$2y\equiv 0$, and hence
\[ Q(y) \equiv y^2 + 2a^2 \equiv 2(c + S^2 + a^2)
\equiv 2\bigg( c + \sum_{i=1}^{s-1} c^{2^i} + \sum_{i=s}^{s+\ell-1} c^{2^i} \bigg)
\equiv 2 f_c^{s+\ell}(0) \pmod{2\omega^2\ZZ[c]} ,\]
where the second congruence is by equation~\eqref{eq:ysimp},
and the third and fourth are by Lemma~\ref{lem:mod2iter} and the definition of $a$.
Since $n | (s+\ell)$, we have $v_2(f_c^{s+\ell}(0))=2^{-s} = v_2(\omega)$ by Lemma~\ref{lem:valFrs}.(1),
and therefore $v_2(Q(y)) = 1 + 2^{-s}$.

If $\gamma\in\QQ_2(c)$, then we would have
\begin{equation}
\label{eq:v2product2}
v_2\big( y - a\gamma \big) + v_2\big( y - a\tilde{\gamma} \big) = v_2(Q(y)) = 1+2^{-s}.
\end{equation}
If $v_2( y - a\gamma) > v_2( y + a\tilde{\gamma} )$,
then as in the proof of Proposition~\ref{prop:root2notQc}, we would have
\[ v_2\big(y-a\gamma\big) > v_2\big(y-a\tilde{\gamma}\big) = v_2\big( a (\gamma-\tilde{\gamma}) \big)
\geq 1, \]
contradicting equation~\eqref{eq:v2product2}.
Similarly, we cannot have $v_2( y - a\gamma) < v_2( y + a\tilde{\gamma} )$, either.

Thus, $v_2( y - a\gamma) = v_2( y + a\tilde{\gamma} )$, so that equation~\eqref{eq:v2product2}
yields $v_2( y - a\gamma)=1/2 + 2^{-s-1}$.
Since $y-a\gamma\in\QQ_2(c)$, it would follow that
the ramification index of $\QQ_2(c)/\QQ_2$ is at least $2^{s+1}$,
contradicting Lemma~\ref{lem:valFrs}.(2), which says that $e(\QQ_2(c)/\QQ_2)=2^s$.
Therefore, we must have $\gamma\not\in\QQ_2(c)$.
\end{proof}


\section{The long-tail strictly preperiodic case ($r>s\geq 2$)}
\label{sec:zeta4longtail}
Throughout this section, we suppose that $0$ is preperiodic under $f(z)=z^2+c$,
and $r>s\geq 0$ are minimal such that $f^r(0)=-f^s(0)$.
We assume that $s\geq 1$, and usually that $s\geq 2$.

\subsection{A fourth root of unity arising in backward orbits}

\begin{lemma}
\label{lem:iroot}
Let $r>s\geq 2$.
Let $x\in\Kbar$, and let $\pm y$ be its two immediate preimages under $f$.
Let $\pm\alpha_1,\ldots , \pm\alpha_{2^{r-1}}$ be the roots of $f^r(z)=y$,
and let $\pm\beta_1,\ldots , \pm\beta_{2^{s-1}}$ be the roots of $f^s(z)=-y$.
Suppose that $f^s(0)\neq -y$, and define
\[ \gamma := \frac{\alpha_1\cdots \alpha_{2^{r-1}} } {\beta_1 \cdots \beta_{2^{s-1}} } .\]
Then $\gamma^2=-1$.
\end{lemma}

\begin{proof}
By Proposition~\ref{prop:key} and the fact that $f^r(0)=-f^s(0)$, we have
\[ \gamma^2 = \frac{f^r(0) - y}{f^s(0)-(-y)} = \frac{-f^s(0) - y}{f^s(0)+y} =  -1.
\qedhere \]
\end{proof}

Recall from Section~\ref{ssec:label} that given a labeling of the tree of preimages $\Orb_f^-(x_0)$,
then for any word $w\in \{a,b\}^i$, 
we write $[w]$ for the element of $f^{-i}(x_0)\subseteq \Kbar$ that appears as the node
with label $w$ in the $i$-th level of the tree.

\begin{lemma}
\label{lem:pickfour}
Let $r>s\geq 2$.
Let $x_0\in K$ not in the forward orbit of $0$,
and choose a primitive $4$-th root of unity $\zeta_4\in\Kbar$.
It is possible to label the tree $T_{\infty}$
of preimages $\Orb_f^-(x_0)$ in such a way that
for every node $x$ of the tree, we have
\begin{equation}
\label{eq:fourprod}
\frac{\prod_{w\in\{a,b\}^{r-1}} [xawa]}{\prod_{w' \in\{a,b\}^{s-1}} [xbw'a]} =
\frac{\prod_{w\in\{a,b\}^{r-1}} [xbwa]}{\prod_{w' \in\{a,b\}^{s-1}} [xaw'a]} = \zeta_4 .
\end{equation}
\end{lemma}

\begin{proof}
We will label the tree inductively, starting from the root point $x_0$.
Label the tree (using the symbols $a$ and $b$) arbitrarily up to level $r$.

For each successive $n\geq r+1$,
suppose that we have labeled $T_{n-1}$ in such a way that
for every node $x$ of $T_{n-1}$ up to level $n-r-2$,
the identity of equation~\eqref{eq:fourprod} holds.
(Note that this condition is vacuously true for $n= r+1$.)
For each node $y$ at level $n-1$, label the two points
of $f^{-1}(y)$ arbitrarily as $ya$ and $yb$. We will now adjust these labels that we have just
applied at the $n$-th level of the tree.

For each node $x$ at level $n-r-1$,
consider the first ratio $\gamma_1$ in equation~\eqref{eq:fourprod}, which is a product
of $2^{r-1}$ nodes sitting $r$ levels above $xa$,
divided by a product of $2^{s-1}$ nodes sitting $s$ levels above $xb$.
By Lemma~\ref{lem:iroot}, we have $\gamma_1^2=-1$,
so $\gamma_1=\pm \zeta_4$.
If $\gamma_1=-\zeta_4$, then exchange the labels of the two
level-$n$ nodes $xaa\ldots aa$ and $xaa\ldots ab$;
otherwise, make no change to the labels of nodes above $xa$.

Similarly, let $\gamma_2$ denote the second ratio in equation~\eqref{eq:fourprod},
which is a product of $2^{r-1}$ nodes sitting $r$ levels above $xb$,
divided by a product of $2^{s-1}$ nodes sitting $s$ levels above $xa$.
Again by Lemma~\ref{lem:iroot}, we have $\gamma_2^2=-1$,
so $\gamma_2=\pm \zeta_4$.
If $\gamma_2=-\zeta_4$, then exchange the labels of the two
level-$n$ nodes $xba\ldots aa$ and $xba\ldots ab$;
otherwise, make no change to the labels of nodes above $xb$.

Having made these (possible) label changes to various nodes at the $n$-th level
of the tree, we have labeled $T_n$ so that for every node $x$
at every level $0\leq \ell \leq n-r-1$ of $T_n$,
the identity of equation~\eqref{eq:fourprod} holds.
Thus, our inductive construction of the desired labeling is complete.
\end{proof}

\subsection{A preliminary result regarding the associated arboreal subgroup}
Lemma~\ref{lem:pickfour} inspires the following result; we state it for the more general case that $r>s\geq 1$, so that
we may apply it in Sections~\ref{sec:Special} and~\ref{sec:shorttail} as well.
Recall the definitions of the sets $M_{r,s,\infty}$ and $B_{r,s,\infty}$
and the function $P_{r,s}$ from Definition~\ref{def:MRSiroot}.

\begin{thm}
\label{thm:MRSiroot}
Fix integers $r>s\geq 1$. Then $M_{r,s,\infty}$ and $B_{r,s,\infty}$ are subgroups
of $\Aut(T_{\infty})$. Moreover, $P_{r,s}:M_{r,s,\infty}\to\ZZ/2\ZZ$
is a homomorphism with kernel $B_{r,s,\infty}$.
\end{thm}

\begin{proof}
\textbf{Step 1}. We claim that for all $\sigma\in M_{r,s,\infty}$,
all $\tau\in \Aut(T_\infty)$, and all nodes $x$ of the tree, we have
\begin{equation}
\label{eq:PRSident}
P_{r,s}(\sigma) + P^a_{r,s}(\tau,x) = P^a_{r,s}(\sigma \tau,x)
\quad\text{and}\quad
P_{r,s}(\sigma) + P^b_{r,s}(\tau,x) = P^b_{r,s}(\sigma \tau,x) .
\end{equation}
Indeed, given such $\sigma$, $\tau$, and $x$, suppose first
that $\Par(\tau,x)=0$. Then $\tau(xa)=\tau(x) a$ and $\tau(xb)=\tau(x) b$,
and hence, writing $P_{r,s}(\sigma)$ as $P^a_{r,s}(\sigma,\tau(x))$, we have
\begin{align*}
P_{r,s}(\sigma) &=
\sum_{w\in\{a,b\}^{r-1}} \Par(\sigma,\tau(x)aw)
+ \sum_{w'\in\{a,b\}^{s-1}} \Par(\sigma,\tau(x)bw')
\\
&= 
\sum_{w\in\{a,b\}^{r-1}} \Par(\sigma,\tau(xa)w)
+ \sum_{w'\in\{a,b\}^{s-1}} \Par(\sigma,\tau(xb)w')
\\
&= 
\sum_{w\in\{a,b\}^{r-1}} \Par(\sigma,\tau(xaw))
+ \sum_{w'\in\{a,b\}^{s-1}} \Par(\sigma,\tau(xbw')).
\end{align*}
Therefore, since we are working in $\ZZ/2\ZZ$, equation~\eqref{eq:sgn22} yields
\begin{align}
\label{eq:PRSkey}
P_{r,s}(\sigma) + P^a_{r,s}(\tau,x) &=
\sum_{w\in\{a,b\}^{r-1}} \big[ \Par(\sigma,\tau(xaw)) + \Par(\tau, xaw) \big]
\notag \\
& \qquad\qquad\qquad
+ \sum_{w'\in\{a,b\}^{s-1}} \big[ \Par(\sigma,\tau(xbw')) + \Par(\tau, xbw') \big]
\notag \\
&= \sum_{w\in\{a,b\}^{r-1}} \Par(\sigma\tau, xaw) 
+ \sum_{w'\in\{a,b\}^{s-1}} \Par(\sigma\tau, xbw')
= P^a_{r,s}(\sigma \tau,x),
\end{align}
proving the first identity of equation~\eqref{eq:PRSident}.
The second follows by a similar calculation, writing $P_{r,s}(\sigma)$ as $P^b_{r,s}(\sigma,\tau(x))$.

The only other possibility is that $\Par(\tau,x)=1$, and hence that
$\tau(xa)=\tau(x) b$ and $\tau(xb)=\tau(x) a$.
Therefore, writing $P_{r,s}(\sigma)$ as $P^b_{r,s}(\sigma,\tau(x))$, we have
\begin{align*}
P_{r,s}(\sigma) &=
\sum_{w\in\{a,b\}^{r-1}} \Par(\sigma,\tau(x)bw)
+ \sum_{w'\in\{a,b\}^{s-1}} \Par(\sigma,\tau(x)aw')
\\
&= 
\sum_{w\in\{a,b\}^{r-1}} \Par(\sigma,\tau(xa)w)
+ \sum_{w'\in\{a,b\}^{s-1}} \Par(\sigma,\tau(xb)w')
\\
&= 
\sum_{w\in\{a,b\}^{r-1}} \Par(\sigma,\tau(xaw))
+ \sum_{w'\in\{a,b\}^{s-1}} \Par(\sigma,\tau(xbw')).
\end{align*}
The first identity of equation~\eqref{eq:PRSident} then follows by exactly the calculation
of equation~\eqref{eq:PRSkey}.
The second follows by a similar calculation, writing $P_{r,s}(\sigma)$ as $P^a_{r,s}(\sigma,\tau(x))$,
and proving our claim.

\medskip

\textbf{Step 2}.
In the notation of Proposition~\ref{prop:Pgroup},
let $\Gamma = \Aut(T_{\infty})$,
let $X$ be the set of nodes of $T_{\infty}$,
let $H=\ZZ/2\ZZ \times \ZZ/2\ZZ$, let $H_0\subseteq H$ be the diagonal subgroup,
and define $P:\Gamma\times X\to H$ by
$P(\sigma, x):=( P_{r,s}^a (\sigma,x), P_{r,s}^b (\sigma,x) )$.

Then $M_{r,s,\infty}$ is precisely the subset $G$ of Proposition~\ref{prop:Pgroup},
and equations~\eqref{eq:PRSident} give precisely hypothesis~\eqref{eq:Phomom}.
Moreover, the identity $e\in\Aut(T_{\infty})$ satisfies $\Par(e,y)=0$ for all nodes $y$,
so that $P_{r,s}^a(e,x) = P_{r,s}^b(e,x)=0$ for all nodes $x$,
and hence $e\in M_{r,s,\infty}$.

Therefore, by Proposition~\ref{prop:Pgroup}, $M_{r,s,\infty}$ is a subgroup of $\Aut(T_{\infty})$,
and $P$ is a homomorphism. Moreover, $P_{r,s}$ is the composition of
$P$ with the unique isomorphism $H_0\to \ZZ/2\ZZ$, so that $P_{r,s}$ is a homomorphism.
By definition, its kernel is $B_{r,s,\infty}$, which is therefore a subgroup of $M_{r,s,\infty}$.
\end{proof}

%

\subsection{The action of Galois}
The following result shows that when $r>s\geq 2$,
the group $M_{r,s,\infty}$ is determined solely by the restriction that a Galois element $\sigma$
must act consistently on every instance of $\zeta_4=\sqrt{-1}$.

\begin{thm}
\label{thm:PRSembed}
Let $r>s\geq 2$.
Let $x_0\in K$ not in the forward orbit of $0$, and choose a 
primitive $4$-th root of unity $\zeta_4\in\Kbar$.
Label the tree $T_{\infty}$ of preimages in $\Orb_f^{-}(x_0)$
as in Lemma~\ref{lem:pickfour}.
Then for any node $x\in\Orb_f^-(x_0)$
and any $\sigma\in G_{\infty}=\Gal(K_{\infty}/K)$, we have
\begin{equation}
\label{eq:PRSembed}
\sigma(\zeta_4) = (-1)^{P^a_{r,s}(\sigma,x)}\zeta_4
= (-1)^{P^b_{r,s}(\sigma,x)}\zeta_4 .
\end{equation}
In particular, the image of $G_{\infty}$ in $\Aut(T_{\infty})$,
induced by its action on $\Orb_f^-(x_0)$ via this labeling,
is contained in $M_{r,s,\infty}$.
Furthermore, if $\zeta_4\in K$, then this Galois image
is contained in $B_{r,s,\infty}$.
\end{thm}

\begin{proof}
It suffices to prove equation~\eqref{eq:PRSembed} for all $x\in\Orb_f^-(x_0)$
and all $\sigma\in G_{\infty}$. Indeed, in that case, equation~\eqref{eq:PRSdef}
would hold for all such $x$ viewed as nodes of the tree, since
$\sigma(\zeta_4)/\zeta_4 \in\{\pm 1\}$ must be independent of $x$.
That is, every $\sigma\in G_{\infty}$ would be in $M_{r,s,\infty}$; and if
$P_{r,s}(\sigma)=0$ for all $\sigma\in G_{\infty}$, then by definition
every $\sigma$ is in $B_{r,s,\infty}$.

Because of the specified labeling of the tree,
for any such $\sigma$ and $x$, we have
\begin{equation}
\label{eq:icross}
\prod_{w\in\{a,b\}^{r-1}} \big[ \sigma(xawa) \big]
= \sigma(\zeta_4) \prod_{w' \in\{a,b\}^{s-1}} \big[ \sigma(xbw'a) \big] .
\end{equation}
Note that for each node $y$ of the tree, we have
\[ \big[ \sigma(ya) \big] = (-1)^{\Par(\sigma,y)} \big[ \sigma(y) a \big] ,\]
and hence equation~\eqref{eq:icross} becomes
\[ (-1)^{P^a_{r,s}(\sigma,x)}\prod_{w\in\{a,b\}^{r-1}} \big[ \sigma(xaw)a \big]
= \sigma(\zeta_4) \prod_{w' \in\{a,b\}^{s-1}} \big[ \sigma(xbw')a \big] . \]
Since each product above is over all $w\in\{a,b\}^{r-1}$ or all $w' \in\{a,b\}^{s-1}$,
this equation becomes either
\[ (-1)^{P^a_{r,s}(\sigma,x)}\prod_{w\in\{a,b\}^{r-1}} \big[ \sigma(x)awa \big]
= \sigma(\zeta_4) \prod_{w' \in\{a,b\}^{s-1}} \big[ \sigma(x)bw'a \big] \]
if $\Par(\sigma,x)=0$, or
\[ (-1)^{P^a_{r,s}(\sigma,x)}\prod_{w\in\{a,b\}^{r-1}} \big[ \sigma(x)bwa \big]
= \sigma(\zeta_4) \prod_{w' \in\{a,b\}^{s-1}} \big[ \sigma(x)aw'a \big] \]
if $\Par(\sigma,x)=1$.
Either way, by Lemma~\ref{lem:pickfour} applied to $\sigma(x)$ in place of $x$, we have
\[ \sigma(\zeta_4) = (-1)^{P^a_{r,s}(\sigma,x)} \zeta_4 .  \]
By similar reasoning, we also have
$\sigma(\zeta_4) = (-1)^{P^b_{r,s}(\sigma,x)} \zeta_4$.
\end{proof}

\section{The special long-tail strictly preperiodic case ($r=3$ and $s=2$)}
\label{sec:Special}
In the notation of the previous section, assume that $r=3$ and $s=2$.
Since $f^3(0)=-f^2(0)$ but $f(0)\neq 0$,
the parameter $c\neq 0$ satisfies
\[ c^4 + 2c^3 + c^2 + c = -(c^2+c) \]
and hence
\begin{equation}
\label{eq:ccubic}
c^3+2c^2+2c+2=0.
\end{equation}
In this specific circumstance, as first observed in \cite[Section~3.8]{PinkPCF},
there are extra subtleties that lead to significantly more involved computations.
In particular, besides $\zeta_4$ arising as an arithmetic combination of elements of $f^{-4}(x)$
as in Lemma~\ref{lem:iroot}, it happens that $\sqrt{2}$ can also be written as
a (much more complicated) arithmetic combination of elements of $f^{-5}(x)$.
We present this formula in Lemma~\ref{lem:root2special}, but we begin with its derivation,
as follows.

\subsection{A second square root arising in backward orbits}

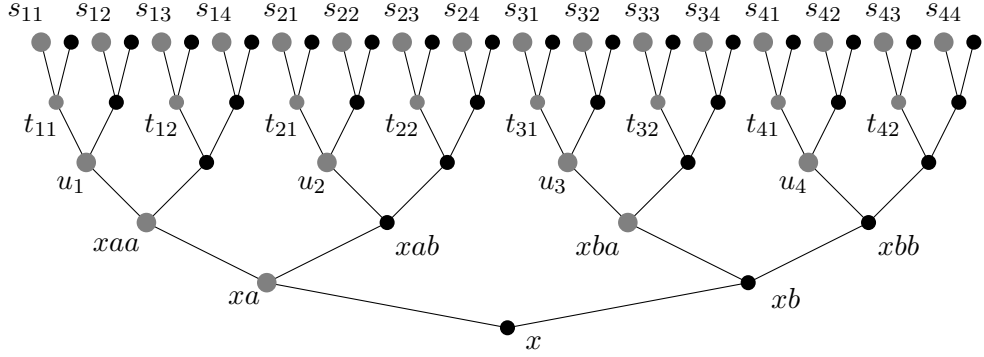
\begin{figure}
\begin{tikzpicture}
\path[draw] (0.2,4) -- (0.4,3.2) -- (0.6,4);
\path[draw] (1,4) -- (1.2,3.2) -- (1.4,4);
\path[draw] (1.8,4) -- (2,3.2) -- (2.2,4);
\path[draw] (2.6,4) -- (2.8,3.2) -- (3,4);
\path[draw] (3.4,4) -- (3.6,3.2) -- (3.8,4);
\path[draw] (4.2,4) -- (4.4,3.2) -- (4.6,4);
\path[draw] (5,4) -- (5.2,3.2) -- (5.4,4);
\path[draw] (5.8,4) -- (6,3.2) -- (6.2,4);
\path[draw] (6.6,4) -- (6.8,3.2) -- (7,4);
\path[draw] (7.4,4) -- (7.6,3.2) -- (7.8,4);
\path[draw] (8.2,4) -- (8.4,3.2) -- (8.6,4);
\path[draw] (9,4) -- (9.2,3.2) -- (9.4,4);
\path[draw] (9.8,4) -- (10,3.2) -- (10.2,4);
\path[draw] (10.6,4) -- (10.8,3.2) -- (11,4);
\path[draw] (11.4,4) -- (11.6,3.2) -- (11.8,4);
\path[draw] (12.2,4) -- (12.4,3.2) -- (12.6,4);
\path[fill,gray] (0.2,4) circle (0.13);
\path[fill] (0.6,4) circle (0.1);
\path[fill,gray] (1,4) circle (0.13);
\path[fill] (1.4,4) circle (0.1);
\path[fill,gray] (1.8,4) circle (0.13);
\path[fill] (2.2,4) circle (0.1);
\path[fill,gray] (2.6,4) circle (0.13);
\path[fill] (3,4) circle (0.1);
\path[fill,gray] (3.4,4) circle (0.13);
\path[fill] (3.8,4) circle (0.1);
\path[fill,gray] (4.2,4) circle (0.13);
\path[fill] (4.6,4) circle (0.1);
\path[fill,gray] (5,4) circle (0.13);
\path[fill] (5.4,4) circle (0.1);
\path[fill,gray] (5.8,4) circle (0.13);
\path[fill] (6.2,4) circle (0.1);
\path[fill,gray] (6.6,4) circle (0.13);
\path[fill] (7,4) circle (0.1);
\path[fill,gray] (7.4,4) circle (0.13);
\path[fill] (7.8,4) circle (0.1);
\path[fill,gray] (8.2,4) circle (0.13);
\path[fill] (8.6,4) circle (0.1);
\path[fill,gray] (9,4) circle (0.13);
\path[fill] (9.4,4) circle (0.1);
\path[fill,gray] (9.8,4) circle (0.13);
\path[fill] (10.2,4) circle (0.1);
\path[fill,gray] (10.6,4) circle (0.13);
\path[fill] (11,4) circle (0.1);
\path[fill,gray] (11.4,4) circle (0.13);
\path[fill] (11.8,4) circle (0.1);
\path[fill,gray] (12.2,4) circle (0.13);
\path[fill] (12.6,4) circle (0.1);
\path[draw] (0.4,3.2) -- (0.8,2.4) -- (1.2,3.2);
\path[draw] (2,3.2) -- (2.4,2.4) -- (2.8,3.2);
\path[draw] (3.6,3.2) -- (4,2.4) -- (4.4,3.2);
\path[draw] (5.2,3.2) -- (5.6,2.4) -- (6,3.2);
\path[draw] (6.8,3.2) -- (7.2,2.4) -- (7.6,3.2);
\path[draw] (8.4,3.2) -- (8.8,2.4) -- (9.2,3.2);
\path[draw] (10,3.2) -- (10.4,2.4) -- (10.8,3.2);
\path[draw] (11.6,3.2) -- (12,2.4) -- (12.4,3.2);
\path[fill,gray] (0.4,3.2) circle (0.1);
\path[fill] (1.2,3.2) circle (0.1);
\path[fill,gray] (2.0,3.2) circle (0.1);
\path[fill] (2.8,3.2) circle (0.1);
\path[fill,gray] (3.6,3.2) circle (0.1);
\path[fill] (4.4,3.2) circle (0.1);
\path[fill,gray] (5.2,3.2) circle (0.1);
\path[fill] (6,3.2) circle (0.1);
\path[fill,gray] (6.8,3.2) circle (0.1);
\path[fill] (7.6,3.2) circle (0.1);
\path[fill,gray] (8.4,3.2) circle (0.1);
\path[fill] (9.2,3.2) circle (0.1);
\path[fill,gray] (10,3.2) circle (0.1);
\path[fill] (10.8,3.2) circle (0.1);
\path[fill,gray] (11.6,3.2) circle (0.1);
\path[fill] (12.4,3.2) circle (0.1);
\path[draw] (0.8,2.4) -- (1.6,1.6) -- (2.4,2.4);
\path[draw] (4,2.4) -- (4.8,1.6) -- (5.6,2.4);
\path[draw] (7.2,2.4) -- (8,1.6) -- (8.8,2.4);
\path[draw] (10.4,2.4) -- (11.2,1.6) -- (12,2.4);
\path[fill,gray] (0.8,2.4) circle (0.13);
\path[fill] (2.4,2.4) circle (0.1);
\path[fill,gray] (4,2.4) circle (0.13);
\path[fill] (5.6,2.4) circle (0.1);
\path[fill,gray] (7.2,2.4) circle (0.13);
\path[fill] (8.8,2.4) circle (0.1);
\path[fill,gray] (10.4,2.4) circle (0.13);
\path[fill] (12,2.4) circle (0.1);
\path[draw] (1.6,1.6) -- (3.2,0.8) -- (4.8,1.6);
\path[draw] (8,1.6) -- (9.6,0.8) -- (11.2,1.6);
\path[fill,gray] (1.6,1.6) circle (0.13);
\path[fill] (4.8,1.6) circle (0.1);
\path[fill,gray] (8,1.6) circle (0.13);
\path[fill] (11.2,1.6) circle (0.1);
\path[draw] (3.2,0.8) -- (6.4,0.2) -- (9.6,0.8);
\path[fill,gray] (3.2,0.8) circle (0.13);
\path[fill] (9.6,0.8) circle (0.1);
\path[fill] (6.4,0.2) circle (0.1);
\node (x) at (6.75,0) {$x$};
\node (y) at (2.9,0.6) {$xa$};
\node (ny) at (10.1,0.6) {$xb$};
\node (w1) at (1.2,1.3) {$xaa$};
\node (nw1) at (5.2,1.3) {$xab$};
\node (w2) at (7.6,1.3) {$xba$};
\node (nw2) at (11.6,1.3) {$xbb$};
\node (u1) at (.6,2.1) {$u_1$};
\node (u2) at (3.8,2.1) {$u_2$};
\node (u3) at (7,2.1) {$u_3$};
\node (u4) at (10.2,2.1) {$u_4$};
\node (t11) at (0.2,2.9) {$t_{11}$};
\node (t12) at (1.8,2.9) {$t_{12}$};
\node (t21) at (3.4,2.9) {$t_{21}$};
\node (t22) at (5.0,2.9) {$t_{22}$};
\node (t31) at (6.6,2.9) {$t_{31}$};
\node (t32) at (8.2,2.9) {$t_{32}$};
\node (t41) at (9.8,2.9) {$t_{41}$};
\node (t42) at (11.4,2.9) {$t_{42}$};
\node (s11) at (0,4.4) {$s_{11}$};
\node (s12) at (0.9,4.4) {$s_{12}$};
\node (s13) at (1.7,4.4) {$s_{13}$};
\node (s14) at (2.5,4.4) {$s_{14}$};
\node (s21) at (3.4,4.4) {$s_{21}$};
\node (s22) at (4.2,4.4) {$s_{22}$};
\node (s23) at (5.0,4.4) {$s_{23}$};
\node (s24) at (5.8,4.4) {$s_{24}$};
\node (s31) at (6.6,4.4) {$s_{31}$};
\node (s32) at (7.4,4.4) {$s_{32}$};
\node (s33) at (8.2,4.4) {$s_{33}$};
\node (s34) at (9.0,4.4) {$s_{34}$};
\node (s41) at (9.8,4.4) {$s_{41}$};
\node (s42) at (10.6,4.4) {$s_{42}$};
\node (s43) at (11.4,4.4) {$s_{43}$};
\node (s44) at (12.2,4.4) {$s_{44}$};
\end{tikzpicture}
\caption{Five levels of the tree above $x$.}
\label{fig:T5special}
\end{figure}

Fix a primitive fourth root of unity $\zeta_4$,
and label the tree $T_{\infty}$ according to Lemma~\ref{lem:pickfour}.
For any node $x$ of the tree, it will be convenient to use
the abbreviated labels shown in Figure~\ref{fig:T5special} for some of the nodes above $x$.
(For example, the node $t_{21}$ is $xabaa$, and the node $-s_{34}$ is $xbabbb$.)
By Lemma~\ref{lem:pickfour}, we have
\begin{equation}
\label{eq:tuzeta}
\frac{t_{11} t_{12} t_{21} t_{22}}{u_3 u_4} = 
\frac{t_{31} t_{32} t_{41} t_{42}}{u_1 u_2} = \zeta_4 ,
\end{equation}
and
\begin{equation}
\label{eq:stzeta}
\frac{s_{11} s_{12} s_{13} s_{14}}{t_{21} t_{22}} =
\frac{s_{21} s_{22} s_{23} s_{24}}{t_{11} t_{12}} =
\frac{s_{31} s_{32} s_{33} s_{34}}{t_{41} t_{42}} =
\frac{s_{41} s_{42} s_{43} s_{44}}{t_{31} t_{32}} = \zeta_4.
\end{equation}

For each word $w\in\{a,b\}^2$, define
\begin{equation}
\label{eq:gammadef}
\gamma_w(x) := [xwaaa][xwaba][xwba] + [xwbaa][xwbba][xwaa],
\end{equation}
which we also write as 
$\gamma_i(x) := s_{i1} s_{i2} t_{i2} + s_{i3} s_{i4} t_{i1}$
for $i=1,2,3,4$.
By Proposition~\ref{prop:key} and equations~\eqref{eq:tuzeta} and~\eqref{eq:stzeta}, we have
\begin{align}
\label{eq:gamma1square}
\gamma_1(x)^2 &=
\big(f^2(0) - u_1\big)(-u_1-c) + 2s_{11} s_{12} s_{13} s_{14} t_{11} t_{12} + \big(f^2(0) + u_1\big)(u_1-c)
\notag
\\
&= 2\big( u_1^2 - cf^2(0) + \zeta_4 t_{11} t_{12} t_{21} t_{22}  \big)
= 2\big( [xaa] -c - cf^2(0) -u_3 u_4 \big)
\\
&= 2\big( [xaa] + (c^2+c+2 - u_3 u_4) \big),
\notag
\end{align}
where the final equality is by equation~\eqref{eq:ccubic}, since $cf^2(0) = c^3+c^2$.
Similarly, we have
\begin{equation}
\label{eq:gamma2square}
\gamma_2(x)^2 = 2\big( [xab] + (c^2+c+2 - u_3 u_4) \big).
\end{equation}
Noting that $[xab]=-[xaa]$,
combining equations~\eqref{eq:gamma1square} and~\eqref{eq:gamma2square} yields
\begin{align}
\label{eq:g1g2}
(\gamma_1(x) \gamma_2(x))^2 &= 4\big(-[xaa]^2 + (c^2+c+2 - u_3 u_4)^2\big)
\notag
\\
&= 4\big( (c-[xa]) + (c^2 + c+2)^2 -2(c^2+c+2) u_3 u_4 + (u_3 u_4)^2 )
\notag
\\
&= 4\big[(c-[xa]) + (3c^2 + 2c+4) - 2(c^2+c+2) u_3 u_4 + \big( f^2(0) - [xb] \big) \big]
\\
&= 4\big[ 4c^2 + 4c + 4 -2(c^2+c+2) u_3 u_4 \big] = 8(c^2+c+1) \big( 2 + c u_3 u_4 \big),
\notag
\end{align}
again using Proposition~\ref{prop:key} and equation~\eqref{eq:ccubic},
as well as the fact that $[xb]=-[xa]$.
A similar computation yields
\[ (\gamma_3(x) \gamma_4(x))^2 = 8(c^2+c+1)( 2 + c u_1 u_2 ), \]
and hence
\begin{equation}
\label{eq:gammaprod}
(\gamma_1(x) \gamma_2(x) \gamma_3(x) \gamma_4(x))^2 =
64 (c^2+c+1)^2 (2+U_a(x)) (2+cU_b(x)),
\end{equation}
where
\begin{equation}
\label{eq:Udef}
U_a(x):= c[xaaa][xaba]=cu_1u_2
\quad\text{and}\quad
U_b(x):= c[xbaa][xbba]=cu_3u_4.
\end{equation}

On the other hand, writing $U_a:=U_a(x)$ and $U_b:=U_b(x)$ for short, we have
\begin{align}
\label{eq:bigdenom}
\big( 2 + U_a + U_b \big)^2 &=
4 + 4\big(U_a + U_b\big) + c^2\big( (u_1 u_2)^2 + (u_3 u_4)^2 \big) + 2U_a U_b
\notag
\\
&= 4 + 4\big(U_a + U_b\big) + c^2\big( f^2(0)-[xa] + f^2(0)-[xb]\big) + 2U_a U_b
\notag
\\
&= 2\big[ \big(2 + c^2 f^2(0) \big) + 2\big(U_a+U_b\big) + U_a U_b \big]
\\
&= 2\big[ 4 + 2U_a + 2U_b + U_a U_b \big] = 2\big(2+U_a\big) \big(2+U_b\big),
\notag
\end{align}
where we have again used Proposition~\ref{prop:key}, equation~\eqref{eq:ccubic},
and the fact that $[xb]=-[xa]$.
The foregoing computations inspire the following result.

\begin{lemma}
\label{lem:root2special}
Let $x_0\in K$ not in the forward orbit of $0$,
and choose both $\sqrt{2}\in\Kbar$
and a primitive $4$-th root of unity $\zeta_4\in\Kbar$.
It is possible to label the tree $T_{\infty}$
of preimages $\Orb_f^-(x_0)$ in such a way that
for every node $x$ of the tree,
equation~\eqref{eq:fourprod} of Lemma~\ref{lem:pickfour} holds for $r=3$ and $s=2$,
and in addition, for every node $x$ of the tree, we have
\begin{equation}
\label{eq:deltaprod}
\frac{\gamma_{aa}(x) \gamma_{ab}(x) \gamma_{ba}(x) \gamma_{bb}(x)}
{4(c^2+c+1) (2 + U_a(x) + U_b(x) )} = \sqrt{2},
\end{equation}
where the quantities $\gamma_w(x)$ and $U_t(x)$ are as in
equations~\eqref{eq:gammadef} and~\eqref{eq:Udef}.
\end{lemma}

\begin{proof}
Choose a labeling of the tree according to Lemma~\ref{lem:pickfour}.
We will now adjust this labeling, working inductively up from the root point $x_0$.

For each $n\geq 0$, suppose that under our current labeling of the tree,
the identities~\eqref{eq:fourprod} and~\eqref{eq:deltaprod} hold
for every node $x$ of the tree up to level $n-1$.
This condition is vacuously true for $n=0$.

For each node $x$ at level $n$, denote by $\delta(x)$
the expression on the left side of equation~\eqref{eq:deltaprod}.
We first claim that $\delta(x)$  makes sense, in that its denominator is nonzero.
Equivalently, by equation~\eqref{eq:bigdenom}, we are claiming
that neither $U_a(x)$ nor $U_b(x)$ is $-2$. However, if $U_a(x)=-2$, then
\[ u_1 u_2 = -\frac{2}{c} = c^2 + 2c+2 \]
by equation~\eqref{eq:ccubic}, and hence
\[ f^2(0)-[xa] = (u_1 u_2)^2 = c^4 + 4c^3 + 8c^2 + 8c + 4 = 2c^2 + 2c = 2f^2(0),\]
again by Proposition~\ref{prop:key} and equation~\eqref{eq:ccubic}.
But then we would have
\[ [xa]=-f^2(0) = f^3(0), \]
and hence $x=f^4(0)$ would be postcritical, which we assumed does not happen.
By this contradiction, we have $U_a(x)\neq -2$.
By a similar argument, we also have $U_b(x)\neq -2$, proving our claim.

Combining equations~\eqref{eq:gammaprod} and~\eqref{eq:bigdenom} yields
\begin{equation}
\label{eq:delta2}
\delta(x)^2 = \frac{ 64 (c^2+c+1)^2 (2+U_a(x)) (2+U_b(x))}{32(c^2+c+1)^2 (2+U_a(x))(2+U_b(x)) } = 2,
\end{equation}
so that $\delta(x)=\pm\sqrt{2}$.
If $\delta(x)=\sqrt{2}$, then make no change to the labeling above $x$.
On the other hand, if $\delta(x)=-\sqrt{2}$,
then swap the labels of the nodes $xaaaa$ and $xaaab$,
and also swap the labels of the nodes $xaaba$ and $xaabb$.
(That is, in the notation of Figure~\ref{fig:T5special},
make two parity changes 4 levels above $x$:
specifically, by switching the labels of $\pm t_{11}$, and also switching the labels of $\pm t_{12}$.
At 5 levels above $x$, this also forces swapping
$s_{11}=[xaaaaa]$ with $s_{12}=[xaaaba]$, as well as $-s_{11}=[xaaaab]$ with $-s_{12}=[xaaabb]$,
and at the same time swapping 
$s_{13}=[xaabaa]$ with $s_{14}=[xaabba]$, and $-s_{13}=[xaabab]$ with $-s_{14}=[xaabbb]$.)

This change reverses the sign of $\gamma_{aa}(x)$ but no other factors in
the expression for $\delta(x)$, which
therefore becomes $+\sqrt{2}$ as desired. In addition, the change does not affect
the value of expression~\eqref{eq:deltaprod} at any other node (in place of $x$)
at or below level $n$ of the tree.

Moreover, the only nodes for which equation~\eqref{eq:fourprod}
might be affected by these swaps are $x$ (for the numerator of the first expression in that equation)
and $xa$ (for the denominator of the second). In both cases, however, the signs of exactly 
two terms of the relevant product change, and hence the value of the full expression is unaffected.
Thus, the change of labels above preserves property that equation~\eqref{eq:fourprod} holds
at every node of the whole tree.

Having made this adjustment above every node $x$ at level $n$, then, our labeling
satisfies equation~\eqref{eq:deltaprod} at every node up to and including level $n$,
and also still satisfies equation~\eqref{eq:fourprod} at every node of the whole tree.
Our inductive construction of the desired labeling is therefore complete.
\end{proof}

\subsection{A preliminary result regarding the associated arboreal subgroup}
We already saw the group $M_{3,2,\infty}$ and its action on $\zeta_4$ via $P_{3,2}$
in Definition~\ref{def:MRSiroot}.
Detecting the action of Galois on $\sqrt{2}$ is much more subtle,
but the following definition turns out to be the appropriate quantity,
in light of Lemma~\ref{lem:root2special}.

\begin{defin}
\label{def:R2root}
Fix a labeling $a,b$ of $T_{\infty}$.
For any word $x$ in the symbols $\{a,b\}$ and any $\sigma\in\Aut(T_{\infty})$,
define
\begin{align}
\label{eq:R2def}
R_{3,2}(\sigma,x) := & \sum_{w\in\{a,b\}^{2}}
\bigg[ \Par(\sigma,xw) + \Par(\sigma,xwb) + \sum_{t\in\{a,b\} } \Par(\sigma,xwat) \bigg]
\\
& + \big[ \Par(\sigma,xaa) + \Par(\sigma,xab) \big] \big[ \Par(\sigma,xba) + \Par(\sigma,xbb) \big]
\in \ZZ/2\ZZ . \notag
\end{align}
Define $\tilde{M}_{3,2,\infty}$ to be the set of all $\sigma\in M_{3,2,\infty}$ for which
\begin{equation}
\label{eq:MR2def}
R_{3,2}(\sigma,x_1) = R_{3,2}(\sigma,x_2) \in\ZZ/2\ZZ
\end{equation}
for all nodes $x_1,x_2$ of $T_{\infty}$. For $\sigma\in \tilde{M}_{3,2,\infty}$,
define $R_{3,2}(\sigma)$ to be this common value.
Finally, define $\tilde{B}_{3,2,\infty} := \{\sigma\in \tilde{M}_{3,2,\infty} \cap B_{3,2,\infty} : R_{3,2}(\sigma)=0 \}$.
\end{defin}

\begin{thm}
\label{thm:R2root}
The sets $\tilde{M}_{3,2,\infty}$ and $\tilde{B}_{3,2,\infty}$ are subgroups
of $M_{3,2,\infty}$. Moreover, $R_{3,2}:\tilde{M}_{3,2,\infty}\to\ZZ/2\ZZ$
is a homomorphism with kernel $\tilde{B}_{3,2,\infty}$.
\end{thm}

To prove Theorem~\ref{thm:R2root}, it will be convenient to define
\begin{equation}
\label{eq:V32def}
V_{3,2}(\sigma,y) :=
\sum_{w\in\{a,b\}^2} \Par(\sigma,yw) + \sum_{w'\in \{a,b\}} \Par(\sigma,yw'),
\end{equation}
for any node $y$ of the tree $T_{\infty}$.
Note that whereas the expressions $P_{3,2}^{a}(\sigma,x)$ and $P_{3,2}^{b}(\sigma,x)$
of Definition~\ref{def:MRSiroot} involve sums one and two levels above $xa$ and $xb$
separately, the sums defining $V_{3,2}(\sigma,y)$ lie one and two levels above the same node $y$.

We will also need the following lemma.

\begin{lemma}
\label{lem:R2root}
Let $\sigma\in M_{3,2,\infty}$, and let $x$ be any node of $T_{\infty}$. Then
\begin{equation}
\label{eq:R2roota}
V_{3,2}(\sigma,xaa) = V_{3,2}(\sigma,xab) = \Par(\sigma,xba) + \Par(\sigma,xbb)
\end{equation}
and
\begin{equation}
\label{eq:R2rootb}
V_{3,2}(\sigma,xba) = V_{3,2}(\sigma,xbb) = \Par(\sigma,xaa) + \Par(\sigma,xab)
\end{equation}
\end{lemma}

\begin{proof}
We will prove half of equation~\eqref{eq:R2roota},
that $V_{3,2}(\sigma,xaa) = \Par(\sigma,xba) + \Par(\sigma,xbb)$.
The proof of the other half, and the proofs of both halves of equation~\eqref{eq:R2rootb}, are similar.

By definition of $P_{3,2}$ (from Definition~\ref{def:MRSiroot}) written as $P^a_{3,2}(\sigma,xa)$,
bearing in mind that we are working in $\ZZ/2\ZZ$, we have
\[ \sum_{w\in\{a,b\}^2} \Par(\sigma,xaaw) =
P_{3,2}(\sigma) + \sum_{w'\in\{a,b\}} \Par(\sigma,xabw'),\]
and hence
\begin{align*} 
V_{3,2}(\sigma,xaa) & = 
\sum_{w\in\{a,b\}^2} \Par(\sigma,xaaw) + \sum_{w'\in \{a,b\}} \Par(\sigma,xaaw') \\
&= P_{3,2}(\sigma) + \sum_{w'\in\{a,b\}} \Par(\sigma,xabw') + \sum_{w'\in \{a,b\}} \Par(\sigma,xaaw') \\
& = P_{3,2}(\sigma) + \sum_{w\in \{a,b\}^2} \Par(\sigma,xaw)
= 2P_{3,2}(\sigma) + \sum_{w'\in \{a,b\}} \Par(\sigma,xbw') \\
& = \Par(\sigma,xba) + \Par(\sigma,xbb),
\end{align*}
where the fourth equality is by writing $P_{3,2}(\sigma) = P^a_{3,2}(\sigma,x)$,
and the fifth is because we are working in $\ZZ/2\ZZ$.
\end{proof}

\begin{proof}[Proof of Theorem~\ref{thm:R2root}]
\textbf{Step~1}.
We claim that for all $\sigma\in\tilde{M}_{3,2,\infty}$, all $\tau\in\Aut(T_{\infty})$,
and all nodes $x$ of the tree, we have
\begin{equation}
\label{eq:Rident}
R_{3,2}(\sigma,\tau(x)) + R_{3,2}(\tau,x) = R_{3,2}(\sigma\tau,x) .
\end{equation}
Indeed, for such $\sigma,\tau, x$, the left side of equation~\eqref{eq:Rident} is
\begin{align}
\label{eq:bigRlines}
\sum_{w\in\{a,b\}^2} &
\big[ \Par(\sigma,\tau(x) w) +\Par(\tau, x w) \big]
+ \sum_{w\in\{a,b\}^2}
\big[ \Par(\sigma,\tau(x)wb) +\Par(\tau, x wb) \big]
\notag
\\
& + \sum_{w\in\{a,b\}^2}
\bigg[ \sum_{t\in\{a,b\} } \Par(\sigma,\tau(x)wat) +\Par(\tau, x wat) \bigg]
\notag
\\
& + \big[ \Par(\sigma,\tau(x)aa) + \Par(\sigma,\tau(x)ab) \big]
\big[ \Par(\sigma,\tau(x)ba) + \Par(\sigma,\tau(x)bb) \big]
\notag
\\
& + \big[ \Par(\tau,xaa) + \Par(\tau,xab) \big] \big[ \Par(\tau,xba) + \Par(\tau,xbb) \big]
\notag
\\
= \sum_{w\in\{a,b\}^2} &
\big[ \Par(\sigma,\tau(xw)) +\Par(\tau, x w) \big]
+ \sum_{w\in\{a,b\}^2}
\big[ \Par(\sigma,\tau(xw)b) +\Par(\tau, x wb) \big]
\notag
\\
& + \sum_{w\in\{a,b\}^2}
\bigg[ \sum_{t\in\{a,b\} } \Par(\sigma,\tau(xw)at) +\Par(\tau, x wat) \bigg]
\\
& + \big[ \Par(\sigma,\tau(xaa)) + \Par(\sigma,\tau(xab)) \big]
\big[ \Par(\sigma,\tau(xba)) + \Par(\sigma,\tau(xbb)) \big]
\notag
\\
& + \big[ \Par(\tau,xaa) + \Par(\tau,xab) \big] \big[ \Par(\tau,xba) + \Par(\tau,xbb) \big]
\notag
\end{align}
By equation~\eqref{eq:sgn22}, the first sum on the right side of equation~\eqref{eq:bigRlines} is 
$\sum_{w\in\{a,b\}^2} \Par(\sigma\tau,xw)$.

We consider the second and third sums together. The contribution of each $w\in \{a,b\}^2$
to these sums is
\begin{equation}
\label{eq:wterm}
\Par(\sigma,\tau(xw)b) + \Par(\tau, xwb)
+ \sum_{t\in\{a,b\} } \big[ \Par(\sigma,\tau(xw)at) +\Par(\tau, x wat) \big] .
\end{equation}
If $\Par(\tau,xw)=0$, so that $\tau(xw)a=\tau(xwa)$ and $\tau(xw)b=\tau(xwb)$,
then expression~\eqref{eq:wterm} is
\begin{align*}
\Par(\sigma,\tau(xwb)) &+ \Par(\tau, xwb)
+ \sum_{t\in\{a,b\} } \big[ \Par(\sigma,\tau(xwat)) +\Par(\tau, x wat) \big]
\\
& = \Par(\sigma\tau,xwb) + \sum_{t\in\{a,b\} } \Par(\sigma\tau,xwat) .
\end{align*}
On the other hand, if $\Par(\tau,xw)=1$, so that $\tau(xw)a=\tau(xwb)$ and $\tau(xw)b=\tau(xwa)$,
then a similar computation shows that expression~\eqref{eq:wterm} is
\[ V_{3,2}(\sigma,\tau(xw)) + \Par(\sigma\tau,xwb) + \sum_{t\in\{a,b\} } \Par(\sigma\tau,xwat), \]
where $V_{3,2}$ is as in equation~\eqref{eq:V32def}.
In either case, then, expression~\eqref{eq:wterm} is
\[ V_{3,2}(\sigma,\tau(xw)) \Par(\tau,xw) + \Par(\sigma\tau,xwb) + \sum_{t\in\{a,b\} } \Par(\sigma\tau,xwat). \]
Hence, the second and third sums of equation~\eqref{eq:bigRlines} together are
\begin{align*}
\label{eq:Rbig2}
\sum_{w\in\{a,b\}^2} &
\bigg[ V_{3,2}(\sigma,\tau(xw)) \Par(\tau,xw)+ \Par(\sigma\tau,xwb) + \sum_{t\in\{a,b\} } \Par(\sigma\tau,xwat) \bigg]
\\
&= \sum_{w\in\{a,b\}^2} \bigg[ \Par(\sigma\tau,xwb) + \sum_{t\in\{a,b\} } \Par(\sigma\tau,xwat) \bigg]
\\
& + \big[ \Par(\sigma,\tau(xba)) + \Par(\sigma,\tau(xbb)) \big] \big[ \Par(\tau,xaa) + \Par(\tau,xab) \big]
\\
& + \big[ \Par(\sigma,\tau(xaa)) + \Par(\sigma,\tau(xab)) \big] \big[ \Par(\tau,xba) + \Par(\tau,xbb) \big]
\end{align*}
where the equality is because by Lemma~\ref{lem:R2root}, we have
\[ V_{3,2}(\sigma,\tau(xa)a) = V_{3,2}(\sigma,\tau(xa)b)
= \Par(\sigma,\tau(xb)a) +\Par(\sigma,\tau(xb)b) \]
and
\[ V_{3,2}(\sigma,\tau(xb)a) = V_{3,2}(\sigma,\tau(xb)b)
= \Par(\sigma,\tau(xa)a) +\Par(\sigma,\tau(xa)b) . \]

Thus, combining all three sums and the last two expressions from equation~\eqref{eq:bigRlines},
we have
\[ R_{3,2}(\sigma,\tau(x)) + R_{3,2}(\tau,x) = R_{3,2}(\sigma\tau, x) + Z_{3,2}(\sigma,\tau,x) \]
where $Z_{3,2}(\sigma,\tau,x)$ is
\begin{align*}
\big[ \Par( & \sigma,\tau(xba)) + \Par(\sigma,\tau(xbb)) \big] \big[ \Par(\tau,xaa) + \Par(\tau,xab) \big]
\\
& + \big[ \Par(\sigma,\tau(xaa)) + \Par(\sigma,\tau(xab)) \big] \big[ \Par(\tau,xba) + \Par(\tau,xbb) \big]
\\
& + \big[ \Par(\sigma,\tau(xaa)) + \Par(\sigma,\tau(xab)) \big]
\big[ \Par(\sigma,\tau(xba)) + \Par(\sigma,\tau(xbb)) \big]
\\
& + \big[ \Par(\tau,xaa) + \Par(\tau,xab) \big] \big[ \Par(\tau,xba) + \Par(\tau,xbb) \big]
\\
& - \big[ \Par(\sigma\tau,xaa) + \Par(\sigma\tau,xab) \big] \big[ \Par(\sigma\tau,xba) + \Par(\sigma\tau,xbb) \big] .
\end{align*}
Applying \eqref{eq:sgn22} to expand each instance of $\Par(\sigma\tau,y)$ in the last line above
as $\Par(\sigma,\tau(y)) + \Par(\tau,y)$, and expanding each of the resulting products, all of the terms cancel.
That is, $Z_{3,2}(\sigma,\tau,x)=0$, proving the claimed identity~\eqref{eq:Rident}.

\medskip

\textbf{Step~2}.
In the notation of Proposition~\ref{prop:Pgroup},
let $\Gamma = M_{3,2,\infty}$,
let $X$ be the set of nodes of $T_{\infty}$,
let $H=H_0=\ZZ/2\ZZ$, and let $P= R_{3,2}$.
Then $\tilde{M}_{r,s,\infty}$ is precisely the subset $G$ of Proposition~\ref{prop:Pgroup},
the identity element $e\in M_{3,2,\infty}$ clearly satisfies $R_{3,2}(e,x)=0$ for all nodes $x$,
and equation~\eqref{eq:Rident} gives hypothesis~\eqref{eq:Phomom}.
Therefore, by Proposition~\ref{prop:Pgroup}, $\tilde{M}_{r,s,\infty}$ is a subgroup of $M_{3,2,\infty}$,
and $R_{3,2}$ is a homomorphism. 
By definition, the kernel of $R_{3,2}$ is $\tilde{B}_{r,s,\infty}$,
which is therefore a subgroup of $\tilde{M}_{r,s,\infty}$.
\end{proof}


\subsection{The action of Galois}
As in previous sections, the following result shows that in our current case, that $r=3$ and $s=2$,
the group $\tilde{M}_{3,2,\infty}$ is determined solely by the restriction that a Galois element
must act consistently on every instance of $\sqrt{-1}$ and of $\sqrt{2}$.

\begin{thm}
\label{thm:R32embed}
Let $x_0\in K$ not in the forward orbit of $0$, and fix a choice of
$\sqrt{2}\in\Kbar$.
Label the tree $T_{\infty}$ of preimages $\Orb_f^{-}(x_0)$
as in Lemma~\ref{lem:pickroot2}.
Then for any node $x\in\Orb_f^-(x_0)$
and any $\sigma\in G_{\infty}=\Gal(K_{\infty}/K)$, we have
\begin{equation}
\label{eq:R32embed}
\sigma(\sqrt{2}) = (-1)^{R_{3,2}(\sigma,x)}\sqrt{2} .
\end{equation}
In particular, the image of $G_{\infty}$ in $\Aut(T_{\infty})$,
induced by its action on $\Orb_f^-(x_0)$ via this labeling,
is contained in $\tilde{M}_{3,2,\infty}$.
Furthermore, if $\zeta_4, \sqrt{2}\in K$, then this Galois image
is contained in $\tilde{B}_{3,2,\infty}$.
\end{thm}

\begin{proof}
\textbf{Step 1}.
We begin with several technical identities needed to prove equation~\eqref{eq:R32embed}.
For every node $x\in\Orb_f^-(x_0)$, word $w\in\{a,b\}^2$, and symbol $t\in\{a,b\}$,
define the quantities $\gamma_w:=\gamma_w(x)$ and $U_t:=U_t(x)$ as in
equations~\eqref{eq:gammadef} and~\eqref{eq:Udef}. Further define
\[ \gamma'_w = \gamma'_{w}(x) := [xw aaa] [xw aba] [xw ba] - [xw baa] [xw bba] [xw aa] .\]
Then, borrowing the notation of Figure~\ref{fig:T5special}, we have
\begin{align*}
\frac{\gamma'_{aa} \gamma'_{ab}}{\gamma_{aa} \gamma_{ab}} &=
\frac{\gamma'_{aa} \gamma_{aa} \gamma'_{ab}\gamma_{ab}}{\gamma_{aa}^2\gamma_{ab}^2}
= \frac{[ (s_{11} s_{12} t_{12})^2 - (s_{13} s_{14} t_{11})^2 ] 
[ (s_{21} s_{22} t_{22})^2 - (s_{23} s_{24} t_{21})^2 ]} { 8(c^2+c+1) (2 + U_b) }
\\
&= \frac{\prod_{i=1}^2[(f^2(0) - u_i)(-c-u_i) - (f^2(0) + u_i)(-c+u_i)]}
{ 8(c^2+c+1) (2 + U_b) }
\\
&= \frac{ [2 (c-f^2(0))u_1] [2 (c-f^2(0))u_2]}{ 8(c^2+c+1) (2 + U_b) }
= \frac{4 c^4 u_1 u_2 }{ 8(c^2+c+1) (2 + U_b) }
= - \frac{U_a }{2 + U_b},
\end{align*}
by equations~\eqref{eq:ccubic} and \eqref{eq:g1g2}, along with Proposition~\ref{prop:key}.
Similarly, we also have
\[ \frac{\gamma'_{ba} \gamma'_{bb}}{\gamma_{ba} \gamma_{bb}} = - \frac{ U_b }{2 + U_a} .\]

Proposition~\ref{prop:key} and equation~\eqref{eq:ccubic} also give us
\begin{align*}
(2-U_a-U_b)(2+U_a+U_b) &= 4 - c^2 (u_1^2 u_2^2 + u_3^2 u_4^2)  - 2 U_a U_b
\\
&= 4 - c^2\big( (f^2(0)- [xa]) + (f^2(0)- [xb]) \big) - 2 U_a U_b 
\\
&= 4 - 2c^2 f^2(0) - 2 U_a U_b =  - 2 U_a U_b,
\end{align*}
using the fact that $[xb]=-[xa]$. Thus, we have
\[ \frac{2 + U_a + U_b}{2  -U_a - U_b} =
\frac{(2 + U_a + U_b)^2}{(2  -U_a - U_b)(2+U_a+U_b)}
= \frac{2(2+U_a)(2+U_b)}{-2U_a U_b} = - \frac{(2+U_a)(2+U_b)}{U_a U_b} , \]
by equation~\eqref{eq:bigdenom}. We also have
\begin{align*}
(2+U_a-U_b)& (2+U_a+U_b) = (2+U_a)^2 - c^2 u_3^2 u_4^2 = (2+U_a)^2 - c^2 (f^2(0)-[xb])
\\
&= (2+U_a)^2 - c^2 \big( f^2(0) + [xa] \big)
= (2+U_a)^2 - c^2 \big( f^2(0) + f^2(0) - u_1^2 u_2^2 \big)
\\
&= 4 -2c^2f^2(0) + 4U_a + 2U_a^2 = 4U_a + 2U_a^2 = 2U_a (2+U_a),
\end{align*}
and hence
\[ \frac{2 + U_a + U_b}{2 + U_a - U_b} =
\frac{(2 + U_a + U_b)^2}{(2  +U_a - U_b)(2+U_a+U_b)}
= \frac{2(2+U_a)(2+U_b)}{2U_a (2+U_a)} = \frac{2+U_b}{U_a} .\]
Reversing the roles of $a$ and $b$, a similar argument yields
\[ \frac{2+U_a+U_b}{2-U_a+U_b} = \frac{2+U_a}{U_b}.\]

\medskip

\textbf{Step 2}.
Given any $\sigma\in G_{\infty}$,
by Theorem~\ref{thm:PRSembed}, we have $\sigma\in M_{3,2,\infty}$.
For each $t\in\{a,b\}$, let $\tilde{t}\in\{a,b\}$ satisfy $\sigma(xt)=\sigma(x)\tilde{t}$.
Let $t'\in\{a,b\}$ be the opposite symbol of $t$ (meaning that $\{t,t'\} = \{a,b\}$),
and let $\tilde{t}'\in\{a,b\}$ be the opposite symbol of $\tilde{t}$. Then
\begin{equation}
\label{eq:Usign}
\sigma\big(U_t(x)\big) = (-1)^{V(t)} c [\sigma(xt)aa][\sigma(xt)ba] = (-1)^{V(t)} U_{\tilde{t}}( \sigma(x) ),
\end{equation}
where the parity $V(t)$ of the sign in equation~\eqref{eq:Usign} is
\[ V(t) := \Par(\sigma,xta) + \Par(\sigma,xtb) =V_{3,2}(\sigma,xt'a) = V_{3,2}(\sigma,xt'b), \]
where the second and third equalities are by Lemma~\ref{lem:R2root}.

Thus, if the sign in equation~\eqref{eq:Usign} is $+1$ (i.e., if the exponent $V(t)$ is even), then
\[ \sigma\big( \gamma_{t'a}(x) \gamma_{t'b}(x) \big)
= \pm \gamma_{\tilde{t}'a}(\sigma(x)) \gamma_{\tilde{t}'b}(\sigma(x)), \]
since $V_{3,2}(\sigma,xt'a)=V_{3,2}(\sigma,xt'b)=0$, and hence
an even number of $-$ signs will appear across the two summands that comprise each $\gamma_w(x)$.
On the other hand, if the sign in equation~\eqref{eq:Usign} is $-1$ (i.e., if the exponent $V(t)$ is odd), then
\[ \sigma\big( \gamma_{t'a}(x) \gamma_{t'b}(x) \big)
= \pm \gamma'_{\tilde{t}'a}(\sigma(x)) \gamma'_{\tilde{t}'b}(\sigma(x)), \]
since an odd number of $-$ signs will appear in each $\gamma_w(x)$.
That is, the sign that $\sigma$ applies to $U_t(x)$ determines whether the $\gamma$ terms above $xt'$
are mapped to plus-or-minus $\gamma$ terms or to plus-or-minus $\gamma'$ terms.

\medskip

\textbf{Step 3}. We now compute $\sigma(\sqrt{2})$ according to the formula~\eqref{eq:deltaprod}
for $\sqrt{2}$. In light of Step~2, we have four cases to consider.

\medskip

\textbf{Case 1}: The parities $V(a)$ and $V(b)$ in equation~\eqref{eq:Usign} are both even, i.e.,
\begin{equation}
\label{eq:parsign1a}
\Par(\sigma,xaa) + \Par(\sigma,xab) = \Par(\sigma,xba) + \Par(\sigma,xbb) = 0 \in\ZZ/2\ZZ .
\end{equation}
In formula~\eqref{eq:R2def} defining $R_{3,2}(\sigma,x)$, then,
the four terms $\Par(\sigma,xw)$ sum to $0$, while the product on the second line
of that formula is also $0$. It remains to consider the twelve terms
$\Par(\sigma,xwb)$ and $\Par(\sigma,xwat)$.

Moreover, as discussed in Step~2,
$\sigma$ maps each $\gamma_w(x)$ to $\pm\gamma_{\tilde{w}}(\sigma(x))$
(where $\sigma(xw)=\sigma(x)\tilde{w}$). In fact, the sign with which 
$\sigma$ maps $\gamma_w(x)$ is the sign with which it maps either one of the terms
in the sum defining $\gamma_w$, namely
\begin{equation}
\label{eq:parsign1b}
\Par(\sigma,xwaa) + \Par(\sigma,xwab) + \Par(\sigma,xwb) .
\end{equation}
Thus, $\sigma(\prod_w \gamma_w) = \pm \prod_w \gamma_w$, where the product
is over all $w\in\{a,b\}^2$, and the parity of the $\pm$ sign is the sum
of expression~\eqref{eq:parsign1b} across all four such words $w$.
But that sum is precisely the sum of the remaining twelve terms of $R_{3,2}(\sigma,x)$.
Combining these observations with the $+$ signs in equation~\eqref{eq:Usign},
it follows that the image of expression~\eqref{eq:deltaprod} under $\sigma$
is $(-1)^{R_{3,2}(\sigma,x)}\sqrt{2}$, as desired.

\medskip

\textbf{Case 2}: The parities in equation~\eqref{eq:Usign} are $V(a)=0$ and $V(b)=1$, i.e.,
\begin{equation}
\label{eq:parsign2a}
\Par(\sigma,xaa) + \Par(\sigma,xab) = 0, \quad \Par(\sigma,xba) + \Par(\sigma,xbb) = 1 \in\ZZ/2\ZZ .
\end{equation}
In formula~\eqref{eq:R2def} defining $R_{3,2}(\sigma,x)$, then,
the four terms $\Par(\sigma,xw)$ sum to $1$, while the product on the second line
of that formula is~$0$.

Because $V(a)=0$, Step~2 shows that $\sigma$ maps $\gamma_{ba}(x)\gamma_{bb}(x)$ to 
$\pm\gamma_{\tilde{b}a}(\sigma(x))\gamma_{\tilde{b}b}(\sigma(x))$, with the parity of the $\pm$ sign
being the sum of expression~\eqref{eq:parsign1b} for $w=ba$ and $w=bb$.
On the other hand, because $V(b)=1$,
Step~2 shows that $\sigma$ maps $\gamma_{aa}(x)\gamma_{ab}(x)$ to 
$\pm\gamma'_{\tilde{a}a}(\sigma(x))\gamma'_{\tilde{a}b}(\sigma(x))$.
This time, the parity of the $\pm$ sign on each individual $\gamma_w\mapsto\gamma'_{\tilde{w}}$ is
\begin{equation}
\label{eq:parsign2b}
\Par(\sigma,xw) + \Par(\sigma,xwaa) + \Par(\sigma,xwab) + \Par(\sigma,xwb),
\end{equation}
where the extra $\Par(\sigma,xw)$ term (as compared with expression~\eqref{eq:parsign1b})
is because swapping the two terms comprising $\gamma'$ introduces a factor of $-1$.

Thus, $\sigma(\gamma_{aa}\gamma_{ab}\gamma_{ba}\gamma_{bb})
= \pm \gamma'_{\tilde{a}a}\gamma'_{\tilde{a}b}\gamma_{\tilde{b}a}\gamma_{\tilde{b}b}$,
where the parity of the sign  is the sum of expression~\eqref{eq:parsign1b} for $w=ba,bb$
plus the sum of expression~\eqref{eq:parsign2b} for $w=aa,ab$.
By equations~\eqref{eq:parsign2a}, this sum is off by $1$ from $R_{3,2}(\sigma,x)$.
On the other hand, equations~\eqref{eq:Usign} and~\eqref{eq:parsign2a} dictate that
$\sigma(U_a)=U_{\tilde{a}}$ but $\sigma(U_b)=-U_{\tilde{b}}$.
We observe that
\begin{multline*}
\frac{ \gamma'_{\tilde{a}a} \gamma'_{\tilde{a}b} \gamma_{\tilde{b}a} \gamma_{\tilde{b}b} }
{ 4(c^2+c+1) (2+U_{\tilde{a}} - U_{\tilde{b}} )} \cdot
\frac{4(c^2+c+1) (2+U_{a} + U_{b})}
{\gamma_{aa}\gamma_{ab}\gamma_{ba}\gamma_{bb}}
\\
= \frac{\gamma'_{\tilde{a}a} \gamma'_{\tilde{a}b} }{\gamma_{\tilde{a}a} \gamma_{\tilde{a}b}}
\cdot \frac{2 + U_{\tilde{a}} + U_{\tilde{b}}}{2+U_{\tilde{a}} - U_{\tilde{b}}}
= \bigg( - \frac{U_{\tilde{a}}}{2 + U_{\tilde{b}} } \bigg) \cdot \bigg( \frac{2 + U_{\tilde{b}} }{U_{\tilde{a}}} \bigg)
=-1,
\end{multline*}
by the identities from Step~1.
Once again, then, the image
of expression~\eqref{eq:deltaprod} under $\sigma$ is $(-1)^{R_{3,2}(\sigma,x)}\sqrt{2}$.

\medskip

\textbf{Case 3}: The parities in equation~\eqref{eq:Usign} are $V(a)=1$ and $V(b)=0$.
This case is the same as Case~2 with the roles of $a$ and $b$ reversed.

\medskip

\textbf{Case 4}: The parities $V(a)$ and $V(b)$ in equation~\eqref{eq:Usign} are both odd, i.e.,
\begin{equation}
\label{eq:parsign4a}
\Par(\sigma,xaa) + \Par(\sigma,xab) = \Par(\sigma,xba) + \Par(\sigma,xbb) = 1 \in\ZZ/2\ZZ .
\end{equation}
In formula~\eqref{eq:R2def} defining $R_{3,2}(\sigma,x)$, then,
the four terms $\Par(\sigma,xw)$ sum to $0$, while the product on the second line
of that formula is $1$.

Because $V(a)=V(b)=1$, $\sigma$ maps each $\gamma_w(x)$ to $\pm\gamma'_{\tilde{w}}(\sigma(x))$,
where the parity of the $\pm$ sign is given by expression~\eqref{eq:parsign2b},
for the reasons discussed in Case~2.
Summing these expressions together is off from $R_{3,2}(\sigma,x)$ by $1$,
because of the aforementioned expression on the second line of formula~\eqref{eq:R2def}.

At the same time, equations~\eqref{eq:Usign} and~\eqref{eq:parsign4a} dictate that
$\sigma(U_a)=-U_{\tilde{a}}$ and $\sigma(U_b)=-U_{\tilde{b}}$.
The relevant quotient is therefore
\begin{multline*}
\frac{ \gamma'_{\tilde{a}a} \gamma'_{\tilde{a}b} \gamma'_{\tilde{b}a} \gamma'_{\tilde{b}b} }
{ 4(c^2+c+1) (2-U_{\tilde{a}} - U_{\tilde{b}} )} \cdot
\frac{4(c^2+c+1) (2+U_{a} + U_{b})}
{\gamma_{aa}\gamma_{ab}\gamma_{ba}\gamma_{bb}}
= \frac{\gamma'_{aa} \gamma'_{ab} }{\gamma_{aa} \gamma_{ab}}
\cdot \frac{\gamma'_{ba} \gamma'_{bb} }{\gamma_{ba} \gamma_{bb}}
\cdot \frac{2 + U_a + U_b}{2 - U_a - U_b}
\\
= \bigg( - \frac{U_a}{2 + U_b} \bigg) \cdot \bigg( - \frac{U_b}{2 + U_a} \bigg)
 \cdot \bigg( - \frac{(2 + U_a)(2+U_b)}{U_a U_b} \bigg)
=-1,
\end{multline*}
by the identities from Step~1.
Yet again, then, the image
of expression~\eqref{eq:deltaprod} under $\sigma$ is $(-1)^{R_{3,2}(\sigma,x)}\sqrt{2}$,
as desired.

\medskip

\textbf{Step 4}.
Having verified equation~\eqref{eq:R32embed} for every $\sigma\in G_{\infty}$,
it follows that $R_{3,2}(\sigma,x_1)=R_{3,2}(\sigma,x_2)$ for all nodes $x_1$, $x_2$ of the tree,
and hence that $\sigma\in \tilde{M}_{3,2,\infty}$.
(Recall that we already observed that $\sigma\in M_{3,2,\infty}$
because of its action on $\zeta_4$, by Theorem~\ref{thm:PRSembed}.)

Finally, if $\zeta_4,\sqrt{2}\in K$, then $\sigma$ must fix both of these elements,
and hence $\sigma\in B_{3,2,\infty}$ with $R_{3,2}(\sigma)=0$.
That is, $\sigma\in \tilde{B}_{3,2,\infty}$.
\end{proof}

\section{The non-Chebyshev short-tail strictly preperiodic case ($s=1$ and $r\geq 3$)}
\label{sec:shorttail}

We now turn to the case that $r>s=1$. That is,
throughout this section, we suppose that $f(z)=z^2+c$ such that
$0$ is not periodic under $f$, but $f^r(0)=-f(0)$ for minimal $r>1$.
We will usually further assume that $r\geq 3$.

\subsection{A square root arising in backwards orbits}

\begin{lemma}
\label{lem:root2}
Let $x\in\Kbar$,
let $\pm y$ be its two preimages under $f$,
let $\pm w_1$ be the two preimages of $y$ under $f$,
and let $\pm w_2$ be the two preimages of $-y$ under $f$.
Setting $m=2^{r-2}$, write
\[ f^{-(r-1)}(w_1) = \{ \pm \alpha_1,\ldots,\pm\alpha_m \} 
\quad\text{and}\quad
f^{-(r-1)}(-w_1) = \{ \pm \alpha'_1,\ldots,\pm\alpha'_m \} . \]
Then
\begin{equation}
\label{eq:prodshorttail}
\prod_{i=1}^m \alpha_i \prod_{i=1}^m \alpha'_i  = \pm w_2 .
\end{equation}
Moreover, if the expression in equation~\eqref{eq:prodshorttail} equals $w_2$,
and if $f^{r-1}(0)\neq -w_2$, then writing
\[ f^{-(r-1)}(-w_2) = \{ \pm \beta_1,\ldots,\pm\beta_m \} ,\]
we have $\gamma^2=2$, where
\begin{equation}
\label{eq:gammashorttail}
\gamma = \frac{\alpha_1 \cdots \alpha_m + \alpha'_1 \cdots \alpha'_m}{\beta_1\cdots \beta_m} .
\end{equation}
\end{lemma}

\begin{proof}
By Proposition~\ref{prop:key} and the fact that $f^r(0)=-f(0)$, we have
\[ \big( \alpha_1 \cdots \alpha_{m} \alpha'_1 \cdots \alpha'_{m} \big)^2
= f^{r}(0) - y = -y - f(0) = w_2^2 . \]
Equation~\eqref{eq:prodshorttail} follows immediately.

Assume for the remainder of the proof that the product in equation~\eqref{eq:prodshorttail}
is $w_2$. Let $\gamma\in K$ be the expression in
equation~\eqref{eq:gammashorttail}. Then by Proposition~\ref{prop:key} again,
we have
\[ \gamma^2 = \frac{(f^{r-1}(0) - w_1) + 2w_2 + (f^{r-1}(0) + w_1)}
{f^{r-1}(0) + w_2}
= \frac{2(f^{r-1}(0)+w_2)}{f^{r-1}(0) + w_2} = 2. \qedhere \]
\end{proof}

\begin{lemma}
\label{lem:pickroot2}
Let $x_0\in K$ not in the forward orbit of $0$,
and fix a choice of $\sqrt{2}\in\Kbar$.
Given a labeling of the tree $T_{\infty}$ of preimages $\Orb_f^-(x_0)$,
then for every node $y$ of the tree, define
\[ E_y = \prod_{w'\in\{a,b\}^{r-2}}[yw'a] \in \Kbar. \]
It is possible to label the tree in such a way that
for every node $x$ of the tree, we have
\begin{equation}\label{eq:shorttaillock}
E_{xaa}E_{xab} = [xba]
\quad\text{and}\quad
E_{xba}E_{xbb} = [xaa],
\end{equation}
and in addition,
\begin{equation}
\label{eq:root2prod}
\frac{E_{xaa}+E_{xab}}{E_{xbb}} =
\frac{E_{xab}-E_{xaa}}{E_{xba}} =
\frac{E_{xba}+E_{xbb}}{E_{xab}} =
\frac{E_{xbb}-E_{xba}}{E_{xaa}} = \sqrt{2} .
\end{equation}
\end{lemma}

\begin{proof}
We will label the tree inductively, starting from the root point $x_0$.
Label the tree (using the symbols $a$ and $b$) arbitrarily up to level $r$.

For each successive $n\geq r+1$,
suppose that we have labeled $T_{n-1}$ in such a way that
for every node $x$ of $T_{n-1}$ up to level $n-r-2$,
the desired identities hold.
(This condition is vacuously true for $n = r+1$.)

For each node $y$ at level $n-1$, label the two points
of $f^{-1}(y)$ arbitrarily as $ya$ and $yb$. We will now adjust these labels,
if necessary.

For each node $x$ at level $n-r-1$, we have
\[ E_{xaa}E_{xab} = \pm [xba]
\quad\text{and}\quad
E_{xba}E_{xbb} = \pm [xaa], \]
by equation~\eqref{eq:prodshorttail} of Lemma~\ref{lem:root2}.
If the first is $-[xba]=[xbb]$, then exchange the labels $ya$ and $yb$ above a single
node $y\in f^{-(r-1)}([xa])$, and the adjusted labeling then satisfies
$E_{xaa}E_{xab} = [xba]$.
If the second is $-[xaa]=[xab]$, make a similar swap above a single node
$y\in f^{-(r-1)}([xb])$.
Equation~\eqref{eq:shorttaillock} now holds for $x$.

Still for the same node $x$ at level $n-r-1$,
having made these adjustments, call the first four quantities in
equation~\eqref{eq:root2prod} $A_1,A_2,B_1,B_2$, respectively.
By equation~\eqref{eq:gammashorttail} of Lemma~\ref{lem:root2}, each is $\pm\sqrt{2}$.
(The negative sign in the numerator of $A_2$ is there because
$E_{xaa} E_{xab}=[xba]=-[xbb]$, so to get the desired value $[xbb]$ on the right side
of equation~\eqref{eq:prodshorttail}, we replace $\alpha'_1 = [xabaa\cdots a]$ by its negative.
Similar reasoning yields the negative sign in the numerator of $B_2$.)
By Proposition~\ref{prop:key}, we have
\[ A_1 A_2 = \frac{ E_{xab}^2 - E_{xaa}^2}{E_{xba} E_{xbb}}
= \frac{(f^{r-1}(0) - [xab]) - (f^{r-1}(0) - [xaa])}{[xaa]}
= \frac{[xaa] + [xaa]}{[xaa]} = 2,\]
since $[xab]=-[xaa]\neq 0$. Since both are square roots of $2$,
it follows that $A_1=A_2$. Similarly, $B_1=B_2$.
We also have
\begin{align*}
\frac{A_1}{B_1} &= \frac{E_{xaa}E_{xab} + E_{xab}^2}{E_{xba}E_{xbb} + E_{xbb}^2}
= \frac{[xba] + (f^{r-1}(0) - [xab])}{[xaa] + (f^{r-1}(0) - [xbb])}
\\
&= \frac{f^{r-1}(0) + [xaa] + [xba]}{f^{r-1}(0) + [xba] + [xaa]} = 1,
\end{align*}
where the second equality is by Proposition~\ref{prop:key}
and (the now verified) equation~\eqref{eq:shorttaillock},
and the third equality is because $[xab]=-[xaa]$ and $[xbb]=-[xba]$.
Thus, we have $A_1=A_2=B_1=B_2=\pm \sqrt{2}$.

If this common value is $-\sqrt{2}$,
then pick $y_1\in f^{-(r-2)}([xaa])$,
and exchange the labels of $y_1 a$ and $y_1 b$.
Also pick $y_2\in f^{-(r-2)}([xab])$,
and exchange the labels of $y_2 a$ and $y_2 b$.
These two switches have the effect of sending each of $E_{xaa}$
and $E_{xab}$ to its negative. Thus, equation~\eqref{eq:shorttaillock}
remains true, but each of $A_1,A_2,B_1,B_2$ is replaced by its negative,
i.e., equation~\eqref{eq:root2prod} holds.

Having made these (possible) label changes to various nodes at the $n$-th level
of the tree, we have labeled $T_n$ so that for every node $x$
at every level $0\leq \ell \leq n-r-1$ of $T_n$,
equations~\eqref{eq:shorttaillock} and~\eqref{eq:root2prod} hold.
Thus, our inductive construction of the desired labeling is complete.
\end{proof}

\subsection{A preliminary result regarding the associated arboreal subgroup}
Recall the functions $P_{r,1}^a$ and $P_{r,1}^b$ 
and the groups $M_{r,1,\infty}$ and $B_{r,1,\infty}$
from Definition~\ref{def:MRSiroot}, with $s=1$.
In particular, we have
\[ P^a_{r,1}(\sigma,x) := 
\Par(\sigma,xb) + \sum_{w\in\{a,b\}^{r-1}} \Par(\sigma,xaw) \in \ZZ/2\ZZ \]
and
\[ P^b_{r,1}(\sigma,x) := 
\Par(\sigma,xa) + \sum_{w\in\{a,b\}^{r-1}} \Par(\sigma,xbw) \in \ZZ/2\ZZ .\]
As in previous sections, Lemma~\ref{lem:pickroot2} inspires the following definition.

\begin{defin}
\label{def:MRroot2}
Fix a labeling $a,b$ of $T_{\infty}$, and fix an integer $r\geq 3$.
For any word $x$ in the symbols $\{a,b\}$ and any $\sigma\in\Aut(T_{\infty})$,
define $R_{r,1}(\sigma,x) \in \ZZ/2\ZZ$ by
\[ R_{r,1}(\sigma,x) := \Par(\sigma,xa)\Par(\sigma,xb) + 
\sum_{w\in\{a,b\}^{r-2}} \Big( \Par\big(\sigma,xabw\big)
+ \Par\big(\sigma,xbbw\big) \Big), \]
and define $\tilde{M}_{r,1,\infty}$ to be the set of all $\sigma\in B_{r,1,\infty}$ for which
\begin{equation}
\label{eq:PR1def2}
R_{r,1}(\sigma,x_1) = R_{r,1}(\sigma,x_2) \in \ZZ/2\ZZ 
\end{equation}
for all nodes $x_1,x_2$ of $T_{\infty}$.
For $\sigma\in \tilde{M}_{r,1,\infty}$, define $R_{r,1}(\sigma)\in\ZZ/2\ZZ$ to be this common value
of $R_{r,1}(\sigma,\cdot)$.
Define $\tilde{B}_{r,1,\infty} := \{\sigma\in \tilde{M}_{r,1,\infty} : R_{r,1}(\sigma)=0 \}$.
\end{defin}

\begin{thm}
\label{thm:MRroot2}
Fix an integer $r\geq 3$. Then $\tilde{M}_{r,1,\infty}$ and $\tilde{B}_{r,1,\infty}$ are subgroups
of $\Aut(T_{\infty})$. Moreover, $R_{r,1}:\tilde{M}_{r,1,\infty}\to\ZZ/2\ZZ$
is a homomorphism with kernel $\tilde{B}_{r,1,\infty}$.
\end{thm}

\begin{proof}
\textbf{Step 1}. We claim that for all $\sigma\in \tilde{M}_{r,1,\infty}$,
all $\tau\in \Aut(T_\infty)$, and all nodes $x$ of the tree, we have
\begin{equation}
\label{eq:PR1halfA}
\Par(\tau,xa) \Par\big(\sigma,\tau(xb) \big)
+ \sum_{w\in \{a,b\}^{r-2}} \Par\big( \sigma,\tau(xab)w \big)
= \sum_{w\in \{a,b\}^{r-2}} \Par\big( \sigma,\tau(xa)bw \big),
\end{equation}
where we continue to work modulo 2.

If $\Par(\tau,xa)=0$, then
$\tau(xab)=\tau(xa)b$, and the first term in equation~\eqref{eq:PR1halfA} is zero,
so our claim holds trivially.

Otherwise, we have $\Par(\tau,xa)=1$, and hence $\tau(xab)=\tau(xa)a$.
Then because we are working modulo~2, the difference between the two
sides of equation~\eqref{eq:PR1halfA} is
\begin{align*}
\Par\big( \sigma,\tau(xb) \big) & +
\sum_{w\in \{a,b\}^{r-2}} \Big( \Par\big( \sigma,\tau(xa)aw \big)
+ \Par\big( \sigma,\tau(xa)bw \big) \Big)
\\
&= \Par\big( \sigma,\tau(xb) \big) + \sum_{w'\in \{a,b\}^{r-1}} \Par\big( \sigma,\tau(xa)w' \big)
\\
&= \begin{cases}
P_{r,1}^a (\sigma,\tau(x)) \text{ if } \Par(\tau,x)=0,
\\
P_{r,1}^b (\sigma,\tau(x)) \text{ if } \Par(\tau,x)=1.
\end{cases}
\end{align*}
Since $\sigma\in B_{r,1,\infty}$, this value is $0$, and hence
equation~\eqref{eq:PR1halfA} holds, proving our claim.

By similar reasoning, we also have
\begin{equation}
\label{eq:PR1halfB}
\Par(\tau,xb) \Par\big(\sigma,\tau(xa) \big)
+ \sum_{w\in \{a,b\}^{r-2}} \Par\big( \sigma,\tau(xbb)w \big)
= \sum_{w\in \{a,b\}^{r-2}} \Par\big( \sigma,\tau(xb)bw \big) .
\end{equation}

\medskip

\textbf{Step 2}. Again with $\sigma,\tau,x$ as in Step 1, we now claim that
\begin{equation}
\label{eq:PR1identKey}
R_{r,1}(\sigma) + R_{r,1}(\tau,x) = R_{r,1}(\sigma \tau,x) .
\end{equation}
Indeed, by equation~\eqref{eq:sgn22} and the definition of $R_{r,1}$, we have

\begin{align*}
R_{r,1}(\sigma\tau,x) &=
\Big( \Par\big(\sigma,\tau(xa) \big) + \Par(\tau,xa) \Big)
\Big( \Par\big(\sigma,\tau(xb) \big) + \Par(\tau,xb) \Big)
\\
& \phantom{\Par} +
\sum_{w\in\{a,b\}^{r-2}} 
\big( \Par(\tau,xabw) + \Par(\tau,xbbw) \big)
\\
& \phantom{\Par} +
\sum_{w\in\{a,b\}^{r-2}} 
\Big( \Par\big(\sigma,\tau(xabw)\big) + \Par\big(\sigma,\tau(xbbw)\big)  \Big)
\\
&= R_{r,1}(\tau,x) + \Par\big(\sigma,\tau(xa) \big) \Par\big(\sigma,\tau(xb) \big)
\\
& \phantom{\Par} +
\Par(\tau,xa) \Par\big(\sigma,\tau(xb) \big)
+ \sum_{w\in \{a,b\}^{r-2}} \Par\big( \sigma,\tau(xab)w \big)
\\
& \phantom{\Par} +
\Par(\tau,xb) \Par\big(\sigma,\tau(xa) \big)
+ \sum_{w\in \{a,b\}^{r-2}} \Par\big( \sigma,\tau(xbb)w \big)
\\
&= R_{r,1}(\tau,x) + \Par\big(\sigma,\tau(xa) \big) \Par\big(\sigma,\tau(xb) \big)
\\
& \phantom{\Par} + \sum_{w\in \{a,b\}^{r-2}}
\Big( \Par\big( \sigma,\tau(xa)bw \big) + \Par\big( \sigma,\tau(xb)bw \big) \Big)
\\
&= R_{r,1}(\tau,x) + \Par\big(\sigma,\tau(x)a \big) \Par\big(\sigma,\tau(x)b \big)
\\
& \phantom{\Par} + \sum_{w\in \{a,b\}^{r-2}}
\Big( \Par\big( \sigma,\tau(x)abw \big) + \Par\big( \sigma,\tau(x)bbw \big) \Big)
\\
&=R_{r,1}(\tau,x) +R_{r,1}\big(\sigma,\tau(x)\big),
\end{align*}
where the second equality is by grouping the terms of $R_{r,1}(\tau,x)$
and rearranging the sums; the third is by equations~\eqref{eq:PR1halfA}
and~\eqref{eq:PR1halfB}; and the fourth is by observing that the two-element
sets $\{\tau(xa),\tau(xb)\}$ and $\{\tau(x)a,\tau(x)b\}$ coincide.
Thus, because $R_{r,1}(\sigma,\tau(x))=R_{r,1}(\sigma)$, we have proven
equation~\eqref{eq:PR1identKey}.

\medskip

\textbf{Step 3}.
Once again, we will finish our proof via Proposition~\ref{prop:Pgroup}.
Let $\Gamma:=B_{r,1,\infty}$, which we know to be a group
by Theorem~\ref{thm:MRSiroot}.
Let $H=H_0:=\ZZ/2\ZZ$, and let $P(\sigma,x):=R_{r,1}(\sigma,x)$.

Clearly, the identity $e\in\Aut(T_{\infty})$ satisfies $R_{r,1}(e,x)=0$ for all nodes $x$,
so that $e\in \tilde{M}_{r,1,\infty}$, which is the group $G$ of Proposition~\ref{prop:Pgroup}.
In addition, equation~\eqref{eq:PR1identKey} provides precisely hypothesis~\eqref{eq:Phomom}.
By Proposition~\ref{prop:Pgroup}, then, $\tilde{M}_{r,1,\infty}$ is a subgroup of $B_{r,1,\infty}$,
and $R_{r,1}$ is a homomorphism.
By definition, its kernel is $\tilde{B}_{r,1,\infty}$, which is therefore a subgroup of $\tilde{M}_{r,1,\infty}$.
\end{proof}

\begin{remark}
We stated Definition~\ref{def:MRroot2} and Theorem~\ref{thm:MRroot2} for the case that $s=1$ and $r\geq 3$,
but they also work with $s=1$ and $r=2$. However, in that special case, corresponding to the Chebyshev
map $f(z)=z^2-2$, the associated arboreal Galois groups are much smaller than the groups defined
in Definition~\ref{def:MRroot2}, because many more special values of the base field arise as arithmetic
combinations of backward orbit elements.
\end{remark}

%

\subsection{The action of Galois}
As in previous sections, the following result shows that for $r\geq 3$, the group $\tilde{M}_{r,1,\infty}$
is determined solely by the restrictions of the action of Galois on equations~\eqref{eq:shorttaillock}
and~\eqref{eq:root2prod} of Lemma~\ref{lem:pickroot2}.

\begin{thm}
\label{thm:PR1embed}
Let $x_0\in K$ not in the forward orbit of $0$, and fix a choice of
$\sqrt{2}\in\Kbar$.
Label the tree $T_{\infty}$ of preimages $\Orb_f^{-}(x_0)$
as in Lemma~\ref{lem:pickroot2}.
Then for any node $x\in\Orb_f^-(x_0)$
and any $\sigma\in G_{\infty}=\Gal(K_{\infty}/K)$, we have
\begin{equation}
\label{eq:PR1embed}
P^a_{r,1}(\sigma,x) = P^b_{r,1}(\sigma,x) = 0
\quad\text{and}\quad
\sigma(\sqrt{2}) = (-1)^{R_{r,1}(\sigma,x)}\sqrt{2} .
\end{equation}
In particular, the image of $G_{\infty}$ in $\Aut(T_{\infty})$,
induced by its action on $\Orb_f^-(x_0)$ via this labeling,
is contained in $\tilde{M}_{r,1,\infty}$.
Furthermore, if $\sqrt{2}\in K$, then this Galois image
is contained in $B_{r,1,\infty}$.
\end{thm}

\begin{proof}
It suffices to prove equations~\eqref{eq:PR1embed} for all $x\in\Orb_f^-(x_0)$
and $\sigma\in G_{\infty}$. Indeed, in that case,
we have that $\sigma\in B_{r,1,\infty}$ from Definition~\ref{def:MRSiroot},
and that equation~\eqref{eq:PR1def2} holds for all nodes $x_1,x_2$ of the tree, since
$\sigma(\sqrt{2})/\sqrt{2} \in\{\pm 1\}$ must be independent of $x$.
That is, equations~\eqref{eq:PR1embed} imply that
every $\sigma\in G_{\infty}$ is in $\tilde{M}_{r,1,\infty}$.
Furthermore, if $\sqrt{2}\in K$, then it would follow that
$R_{r,1}(\sigma)=0$ for all $\sigma\in G_{\infty}$, 
and hence that every such $\sigma$ lies in $\tilde{B}_{r,1,\infty}$.

Note that for any $y\in\Orb_f^-(x_0)$, we have
\begin{equation}
\label{eq:sigmaE}
\sigma(E_y) = (-1)^{R(\sigma,y)} E_{\sigma(y)},
\quad\text{where } R(\sigma,y) := \sum_{w\in\{a,b\}^{r-2}} \Par(\sigma,yw) .
\end{equation}
Indeed, equation~\eqref{eq:sigmaE} is immediate from the
definition of $E_y$ in Lemma~\ref{lem:pickroot2}, combined
with the fact that for each node $u$ of the tree, we have
\[ \sigma\big([ua]\big) = (-1)^{\Par(\sigma,u)} \big[ \sigma(u) a \big] .\]

\medskip

\textbf{Step 1}. We begin by proving the first part
of equations~\eqref{eq:PR1embed}.
Given $x\in\Orb_f^-(x_0)$ and $\sigma\in G_{\infty}$,
if $\Par(\sigma,x)=0$, so that $\sigma(xa)=\sigma(x)a$ and $\sigma(xb)=\sigma(x)b$, then 
by the first equation of~\eqref{eq:shorttaillock} and the fact that
the two-element sets
$\{\sigma(xaa),\sigma(xab)\}$ and $\{\sigma(x)aa,\sigma(x)ab\}$ coincide,
we have
\begin{align*}
1 &= \sigma(1) = \frac{\sigma(E_{xaa}) \sigma(E_{xab})}{\sigma( [xba] ) }
= (-1)^{R(\sigma,xaa) + R(\sigma,xab)} \frac{E_{\sigma(x)aa} E_{\sigma(x)ab} }
{(-1)^{\Par(\sigma,xb)} [\sigma(x)ba]}
\\
&= (-1)^{P_{r,1}^a(\sigma,x)} \frac{E_{\sigma(x)aa} E_{\sigma(x)ab} }{[\sigma(x)ba]}
= (-1)^{P_{r,1}^a(\sigma,x)} .
\end{align*}
(Here, the second and fifth equalities are by the choice of labeling of the tree,
the third is by equation~\eqref{eq:sigmaE},
and the fourth is by the definitions of $P_{r,1}^a$ and $R$.)
On the other hand, if $\Par(\sigma,x)=1$, then 
$\sigma(xa)=\sigma(x)b$ and $\sigma(xb)=\sigma(x)a$, and hence
\begin{align*}
1 &= \sigma(1) = \frac{\sigma(E_{xaa}) \sigma(E_{xab})}{\sigma( [xba] ) }
= (-1)^{R(\sigma,xaa) + R(\sigma,xab)} \frac{E_{\sigma(x)ba} E_{\sigma(x)bb} }
{(-1)^{\Par(\sigma,xb)} [\sigma(x)aa]}
\\
&= (-1)^{P_{r,1}^a(\sigma,x)} \frac{E_{\sigma(x)ba} E_{\sigma(x)bb} }{[\sigma(x)aa]}
= (-1)^{P_{r,1}^a(\sigma,x)} .
\end{align*}
Either way, then, we have $P_{r,1}^a(\sigma,x)=0$ in $\ZZ/2\ZZ$.
By similar reasoning applied to the second equation of~\eqref{eq:shorttaillock},
we also have $P_{r,1}^b(\sigma,x)=0$,
verifying the first portion of equations~\eqref{eq:PR1embed}.

\medskip

\textbf{Step 2}. Given $x\in\Orb_f^-(x_0)$ and $\sigma\in G_{\infty}$,
then by equation~\eqref{eq:sigmaE} and property~\eqref{eq:root2prod}
of the specified labeling of the tree, we have 
\[ (-1)^{R(\sigma,xaa)} E_{\sigma(xaa)} + (-1)^{R(\sigma,xab)} E_{\sigma(xab)}
=\sigma(\sqrt{2}) \cdot (-1)^{R(\sigma,xbb)} E_{\sigma(xbb)} .\]
Multiplying both sides of this equation by $(-1)^{R(\sigma,xab)} / E_{\sigma(xbb)}$,
and applying the identity $P^a_{r,1}=0$ proven in Step~1, i.e., 
$R(\sigma,xaa)+R(\sigma,xab)=\Par(\sigma,xb)$, we obtain
\begin{equation}
\label{eq:root2sigma}
(-1)^{R(\sigma,xab)+R(\sigma,xbb)} \sigma(\sqrt{2}) = 
\frac{ E_{\sigma(xab)} + (-1)^{\Par(\sigma,xb)} E_{\sigma(xaa)} } {E_{\sigma(xbb)} }.
\end{equation}
We consider three cases. First, if $\Par(\sigma,xb)=0$,
then equation~\eqref{eq:root2sigma} becomes
\[ (-1)^{R(\sigma,xab)+R(\sigma,xbb)} \sigma(\sqrt{2}) = 
\frac{ E_{\sigma(xa)a} + E_{\sigma(xa)b} } {E_{\sigma(xb)b} } = \sqrt{2}, \]
by considering the first expression in equation~\eqref{eq:root2prod}
if $\Par(\sigma,x)=0$, or the third if $\Par(\sigma,x)=1$.
Here, we have used the fact that the two-element sets 
$\{\sigma(xaa),\sigma(xab)\}$ and $\{\sigma(xa)a,\sigma(xa)b\}$ coincide.

Second, if $\Par(\sigma,xa)=0$ and $\Par(\sigma,xb)=1$,
then equation~\eqref{eq:root2sigma} becomes
\[ (-1)^{R(\sigma,xab)+R(\sigma,xbb)} \sigma(\sqrt{2}) = 
\frac{ E_{\sigma(xa)b} - E_{\sigma(xa)a} } {E_{\sigma(xb)a} } = \sqrt{2}, \]
by considering the second expression in equation~\eqref{eq:root2prod}
if $\Par(\sigma,x)=0$, or the fourth if $\Par(\sigma,x)=1$.

Third, if $\Par(\sigma,xa)=\Par(\sigma,xb)=1$,
then equation~\eqref{eq:root2sigma} becomes
\[ (-1)^{R(\sigma,xab)+R(\sigma,xbb)} \sigma(\sqrt{2}) = 
\frac{ E_{\sigma(xa)a} - E_{\sigma(xa)b} } {E_{\sigma(xb)a} } = - \sqrt{2}, \]
again by considering the second or fourth expression in equation~\eqref{eq:root2prod}.

Recalling that
\[ R_{r,1}(\sigma,x)=\Par(\sigma,xa)\Par(\sigma,xb) + R(\sigma,xab)+R(\sigma,xbb) ,\]
then in all three cases, we have
$\sigma(\sqrt{2}) = (-1)^{R_{r,1}(\sigma,x)}\sqrt{2}$, as desired.
\end{proof}


\section{The geometric Galois groups}
\label{sec:PinkGroupsGeom}

In this section, we show that our groups 
$B_{r,s,\infty}$ and $\tilde{B}_{r,s,\infty}$ coincide with Pink's geometric groups
$G^{\Pink}_{r,s,\infty}\cong G^{\geom}$,
which we introduced in Section~\ref{ssec:PinkSummary}.
(We will relate his larger arithmetic group $G^{\arith}$
to our larger groups $M_{r,s,\infty}$,  $\tilde{M}_{r,s,\infty}$ in Section~\ref{sec:obtain}.)

\subsection{Restricting to finite subtrees}
For $0\leq m\leq n\leq \infty$, define homomorphisms
\[ \res_{n,m}:\Aut(T_n)\rightarrow \Aut(T_m) \]
given by restricting elements of $\Aut(T_n)$ to the $m$-th level of the tree.
Note that $\Aut(T_\infty)$ is naturally isomorphic to the inverse limit of the finite groups
$\Aut(T_n)$ with respect to these restriction homomorphisms.
In particular, these homomorphisms induce a natural profinite topology on $\Aut(T_\infty)$
from the discrete topology on each of the finite groups $\Aut(T_n)$.

\begin{lemma}
\label{lem:closedsubgps}
The groups $B_{r,s,\infty}$, $M_{r,s,\infty}$, $\tilde{B}_{3,2,\infty}$, $\tilde{M}_{3,2,\infty}$,
$\tilde{B}_{r,1,\infty}$, and $\tilde{M}_{r,1,\infty}$ are closed subgroups of $\Aut(T_\infty)$.
\end{lemma}

\begin{proof}
Let $X$ be the set of nodes of $T_{\infty}$.
Each of the subgroups $G$ in question
is defined by relations involving functions $P:\Aut(T_{\infty})\times X\to\ZZ/2\ZZ$
for which $P(\sigma,x)$ depends only on $\Par(\sigma,y)$ for nodes $y$ lying a bounded number of levels above $x$.
(Here, $P$ is any of $P^a_{r,s}$ or $P^b_{r,s}$ or $R_{3,2}$ or $R_{r,1}$.)
By definition of the profinite topology, then, for each $x\in X$, the function
$\sigma\mapsto P(\sigma,x)$ is continuous, and thus each of the equations defining the subgroup $G$
(for a single $x\in X$ or a pair $x_1,x_2\in X$) defines a closed subset of $\Aut(T_{\infty})$.
The subgroup $G$ is therefore an intersection of closed sets and hence is closed.
\end{proof}

For each integer $n\geq 1$, define
\[ B_{r,s,n}:= \res_{\infty,n}( B_{r,s,\infty} )
\quad\text{and}\quad
G^{\Pink}_{r,s,n}:=\res_{\infty,n}( G^{\Pink}_{r,s,\infty}) , \]
which are subgroups of $\Aut(T_n)$.
Similarly define the subgroups $\tilde{B}_{r,s,n}:= \res_{\infty,n}( \tilde{B}_{r,s,\infty} )$ for $s=1$ and for $(r,s)=(3,2)$.
By \cite[Proposition 3.3.3]{PinkPCF}, we have
\begin{equation}
\label{eq:pinksize}
\log_2 \big|G^{\Pink}_{r,s,n} \big|=
\begin{cases}
2^n-1 & \text{if  $n\leq r$,}\\
n+1 & \text{if  $s=1$ and $r=2$ and $n\geq 2$,}\\
2^n-3\cdot 2^{n-r}+2 & \text{if  $s=1$ and $r\geq 3$ and $n\geq r$,}\\
2^n-5\cdot 2^{n-4}+2 & \text{if  $s=2$ and $r=3$ and $n\geq 4$,}\\
2^n-2^{n-r+1}+1 &  \text{if $s\geq 2$ and $r\geq 4$ and $n\geq r$.}
\end{cases}
\end{equation}

\subsection{Pink's generators}

Pink defines his subgroups $G^{\Pink}_{r,s,\infty}\subseteq\Aut(T_\infty)$ in terms of topological generators,
as follows.
Let $e$ denote the identity element of $\Aut(T_{\infty})$, and
let $\mapt\in \Aut(T_\infty)$ be the automorphism that swaps the two subtrees rooted
at the two nodes $a,b$ at level~1. That is, for all nodes $x$ of the tree, we have
\[ \Par(e,x)=0 \quad \text{and} \quad
\Par(\mapt,x) = \begin{cases}
1 & \text{ if } x=x_0, \\
0 & \text{ otherwise}.
\end{cases} \]
For any $\mapa,\mapb\in\Aut(T_\infty)$, write $(\mapa,\mapb)\in\Aut(T_{\infty})$
for the automorphism that fixes the nodes $a,b$ at level~1, and acts as $\mapa$
on the subtree rooted at $a$, and as $\mapb$ on the subtree rooted at $b$.
That is, $\Par( (\mapa,\mapb), x_0)=0$, and for any word $w$ in the symbols $\{a,b\}$, we have
\[ \Par\big( (\mapa,\mapb), aw \big) = \Par(\mapa,w)
\quad \text{and}\quad \Par\big( (\mapa,\mapb), bw \big) = \Par(\mapb,w) .\]

Pink's subgroup $G^{\Pink}_{r,s,\infty}$ is the closure of the subgroup generated by elements
$\mapa_1,\ldots,\mapa_r\in\Aut(T_{\infty})$ given by the recursive relations
\begin{equation}
\label{eq:PinkTailDef}
\mapa_1 = \mapt, \qquad
\mapa_{s+1} = (\mapa_s, \mapa_r),
\qquad\text{and}\qquad
\mapa_i = (\mapa_{i-1}, 1) \quad \text{for } i\neq 1, s+1.
\end{equation}

Equivalently, we may understand Pink's generators as follows.
For $1\leq i\leq s$, we have
\[ \Par(\mapa_i,y) = \begin{cases}
1 & \text{ if } y = a^{i-1}, \\
0 & \text{ otherwise},
\end{cases} \]
where we write $w^j$ for the word $ww\cdots w$ consisting of $j$ consecutive copies of the word $w$.
On the other hand, for $i\geq s+1$ the recursive definition comes into play.
Define $\ell:=r-s\geq 1$, and
let $w_{\ell}$ be the word of length $\ell$ given by $w_{\ell}:= b a^{\ell-1}$.
A straightforward induction shows that for each $i=s+1,\ldots, r$,
the automorphism $\mapa_i$ of equation~\eqref{eq:PinkTailDef} is given by
\begin{equation}
\label{eq:Parmapa2}
\Par(\mapa_i, w) = \begin{cases}
1 & \text{ if } w = a^{i-s-1} w_\ell^n a^s \text{ for some } n\geq 0, \\
0 & \text{ otherwise}.
\end{cases}
\end{equation}

By Lemma~\ref{lem:closedsubgps}, to prove that one of our groups $B_{\infty}=B_{r,s,\infty}$
or $B_{\infty}=\tilde{B}_{r,s,\infty}$
contains Pink's group $G^{\Pink}_{r,s,\infty}$, it suffices to show that each generator $\mapa_i$ lies in $B_{\infty}$.
Then, to prove that in fact $G^{\Pink}_{r,s,\infty}=B_\infty$, it remains to show that
that $|G^{\Pink}_{r,s,n}| = |B_n|$ for each $n$.
We prove precisely these claims in the various cases in the next few sections.

\subsection{The general long tail case}

\begin{prop}
\label{prop:PinkSubTail}
For any integers $r>s\geq 1$, Pink's subgroup $G^{\Pink}_{r,s,\infty}$ is contained in $B_{r,s,\infty}$.
\end{prop}

\begin{proof}
By Lemma \ref{lem:closedsubgps}, it suffices to prove that each of the generators $\mapa_1,\ldots,\mapa_r$ of Pink's group belongs to $B_{r,s,\infty}$. 
As in equation~\eqref{eq:Parmapa2}, let $\ell:=r-s\geq 1$, and let $w_{\ell}:= b a^{\ell-1}$.

Fix a node $x$ of the tree, and an integer $i\in \{ 1,\ldots, r\}$.
Recalling Definition~\ref{def:MRSiroot},
we must show that $P_{r,s}^t(\mapa_i,x)\equiv 0\pmod{2}$
for both $t=a$ and $t=b$.
We consider three cases.

First, suppose $i\leq s$. Then $\Par(\mapa_i,w)=0$ for all words $w$ of length at least $s$.
In particular, all of the terms appearing in the definition of $P_{r,s}^a(\mapa_i,x)$
and $P_{r,s}^b(\mapa_i,x)$ are zero, as desired.

Second, suppose that $i=s+1+j$ for some $0\leq j \leq \ell-1$,
and that there is some node $y$ exactly $s$ levels above $x$
for which $\Par(\mapa_i, y) =1$. Then according to equation~\eqref{eq:Parmapa2},
we must have $y=a^j w_{\ell}^n a^s$ for some $n\geq 0$,
and hence $x=a^j w_{\ell}^n$. Thus, $y=xaa^{s-1}$ is the only node $s$ levels above $x$
for which $\Par(\mapa_i, y) =1$, and $y':=xw_{\ell}a^s = xb a^{r-1}$ 
is the only node $r$ levels above $x$ for which $\Par(\mapa_i, y') =1$. Therefore,
\[ P_{r,s}^a(\mapa_i,x) = 0 + 0 = 0,
\quad\text{and}\quad
P_{r,s}^b(\mapa_i,x) = 1 + 1 \equiv 0 \pmod{2}, \]
as desired.

Third, suppose that $i=s+1+j$ for some $0\leq j \leq \ell-1$,
and that $\Par(\mapa_i, y) =0$ for all nodes $y$ that are $s$ levels above $x$.
Again by equation~\eqref{eq:Parmapa2}, $x$ cannot be of the form $a^j w_{\ell}^n$ for any $n\geq 0$,
and therefore
we must have $\Par(\mapa_i, y') =0$ for all nodes $y'$ that are $r$ levels above $x$,
since the word $w_{\ell} a^s$ has length $r$. Thus, 
all of the terms appearing in the definition of $P_{r,s}^a(\mapa_i,x)$
and $P_{r,s}^b(\mapa_i,x)$ are zero, as desired.
\end{proof}

\begin{thm}
\label{thm:PinkLongTail}
For all integers $r>s\geq 2$ with $r\geq 4$, we have $G^{\Pink}_{r,s,\infty}=B_{r,s,\infty}$.
\end{thm}

\begin{proof}
For each integer $n\geq 1$, define
\[ B'_{r,s,n}:= \{ \sigma\in \Aut(T_n) : \forall m<n-r \text{ and } \forall x\in\{a,b\}^m, \,
P^a_{r,s}(\sigma,x)=P^b(\sigma,x)=0\},\]
where we recall $P^a_{r,s}$ and $P^b_{r,s}$ take values in $\ZZ/2\ZZ$.
By Proposition~\ref{prop:PinkSubTail} and Definition~\ref{def:MRSiroot}, we have
\[ G^{\Pink}_{r,s,n} \subseteq B_{r,s,n} \subseteq B'_{r,s,n} . \]
Thus, it suffices to show that $|B'_{r,s,n}| \leq |G^{\Pink}_{r,s,n}|$ for all $n\geq 1$.

We proceed by induction on $n$.
For $n\leq r$, we have $B'_{r,s,n}= \Aut(T_n)$, and hence
\[ \log_2\big|B'_{r,s,n}\big|=\log_2|\Aut(T_n)|= 2^n-1 = \log_2 \big|G^{\Pink}_{r,s,n}\big|, \] 
recalling formula~\eqref{eq:pinksize} for $\log_2 |G^{\Pink}_{r,s,n}|$.

Now let $n\geq r+1$, and suppose $|B'_{r,s,n-1}|\leq |G^{\Pink}_{r,s,n-1}|$. 
Let $S_{r,s,n}$ denote the kernel of the map $\res_{n,n-1} : B'_{r,s,n}\rightarrow B'_{r,s,n-1}$, so that
\[ \big|B'_{r,s,n}\big| = \big|\res_{n,n-1}(B'_{r,s,n})\big|\cdot \big|S_{r,s,n}\big|
\leq \big|B'_{r,s,n-1}\big|\cdot \big|S_{r,s,n}\big| .\]

We now compute $|S_{r,s,n}|$.
Let $Y_n$ denote the kernel of
$\res_{n,n-1} : \Aut(T_n)\to\Aut(T_{n-1})$. Then $S_{r,s,n}$ is the set of $\sigma\in Y_n$ for which
\begin{equation}
\label{eq:condition_B_lt}
P^a_{r,s}(\sigma,x)=P^b_{r,s}(\sigma,x)=0\in \ZZ/2\ZZ
\end{equation}
for all $0\leq m < n-r$ and all $x\in\{a,b\}^m$.

For $\sigma\in Y_n$, condition~\eqref{eq:condition_B_lt} is trivially satisfied
for every node $x$ at levels $m<n-r-1$. Thus, we only need to consider nodes $x$
at level $m=n-r-1$. For such $x$, condition~\eqref{eq:condition_B_lt} holds if and only if
\begin{equation}
\label{eq:longtaileven}
\sum_{w\in \{a,b\}^{r-1}} \Par(\sigma,xaw) \quad\text{and }
\sum_{w\in \{a,b\}^{r-1}} \Par(\sigma,xbw) \quad\text{are both even}.
\end{equation}

To determine whether a given $\sigma\in Y_n$ belongs to $S_{r,s,n}$,
for each node $x$ at level $n-r-1$, the values of 
$\Par(\sigma,xaw)$ for $w\in \{a,b\}^{r-1}$ can be arbitrary except for $\Par(\sigma,xab^{r-1})$, 
which must be chosen so that the first sum in condition~\eqref{eq:longtaileven} is even.
Similarly, the values of
$\Par(\sigma,xbw)$ for $w\in \{a,b\}^{r-1}$ can be arbitrary except for $\Par(\sigma,xbb^{r-1})$, 
which must be chosen so that the second sum is even, yielding two parity restrictions at the node $x$.
Since there are $2^{n-r-1}$ nodes $x$ at level $n-r-1$,
each with independent such parity restrictions,
there are $2\cdot 2^{n-r-1}$ parity restrictions in total.
There are no restrictions on
the parities at the remaining $2^{n-1}-2^{n-r}$ nodes at level $n-1$, and therefore
\[\log_2\big|S_{r,s,n}\big|=2^{n-1}-2^{n-r}.\]

By equation~\eqref{eq:pinksize}, 
\begin{align*}
\log_2 \big|G^{\Pink}_{r,s,n}\big|-\log_2 \big|G^{\Pink}_{r,s,n-1}\big|
& = (2^n-2^{n-r+1}+1) -(2^{n-1}-2^{n-1-r+1}+1)\\
&=2^{n-1}-2^{n-r} =\log_2\big|S_{r,s,n}\big|.
\end{align*}
It follows that $\log_2 |G^{\Pink}_{r,s,n-1}| + \log_2|S_{r,s,n}| = \log_2 |G^{\Pink}_{r,s,n}|$, and hence
\[ \big|B'_{r,s,n}\big| \leq \big|B'_{r,s,n-1}\big| \cdot \big|S_{r,s,n}\big|
\leq  \big|G^{\Pink}_{r,s,n-1}\big| \cdot \big|S_{r,s,n}\big| = \big|G^{\Pink}_{r,s,n}\big|, \]
as desired.
\end{proof}

\subsection{The special long-tail case}
\label{ssec:PinkSpecial}
Recall from Section~\ref{sec:Special} that when $r=3$ and $s=2$,
we defined a proper subgroup $\tilde{B}_{3,2,\infty}$ of $B_{3,2,\infty}$
in Definition~\ref{def:R2root}. In this case, it is this smaller group that
realizes $G^{\geom}$. To see this, first observe that Pink's generators in this case are
\begin{equation}
\label{eq:PinkSpecDef}
\mapa_1 = \mapt, \qquad
\mapa_2 = (\mapt, 1),
\qquad\text{and}\qquad
\mapa_3 = (\mapa_2, \mapa_3).
\end{equation}
In particular, we have $\Par(\mapa_1,x)=1$ only for $x=x_0$,
and $\Par(\mapa_2,x)=1$ only for $x=a$, while
\begin{equation}
\label{eq:Parmapa3}
\Par(\mapa_3, w) = \begin{cases}
1 & \text{ if } w = b^n aa \text{ for some } n\geq 0, \\
0 & \text{ otherwise}.
\end{cases}
\end{equation}

\begin{prop}
\label{prop:PinkSubSpec}
Pink's subgroup $G^{\Pink}_{3,2,\infty}$ is contained in $\tilde{B}_{3,2,\infty}$.
\end{prop}

\begin{proof}
By Proposition~\ref{prop:PinkSubTail}, we already have
$G^{\Pink}_{3,2,\infty}\subseteq B_{3,2,\infty}$. It suffices to show that
each of $\mapa_1,\mapa_2,\mapa_3$ belongs to $\tilde{B}_{3,2,\infty}$.

Fix a node $x$ of the tree. We must show that $R_{3,2}(\mapa_i,x)\equiv 0\pmod{2}$
for each of $i=1,2,3$.
For $i=1,2$, this is trivial, because the formula for $R_{3,2}(\mapa_i,x)$
from Definition~\ref{def:R2root} involves $\Par(\mapa_i,y)$ only for nodes $y$
lying at least two levels above $x$, for which $\Par(\mapa_i,y)=0$.
It remains to consider $\mapa_3$, via two cases.

First, suppose $x=b^n$ for some $n\geq 0$,
and consider equation~\eqref{eq:R2def} defining $R_{3,2}(\sigma,x)$.
According to equation~\eqref{eq:Parmapa3}, for any word $w$ of length $2$,
the first term $\Par(\mapa_3,xw)$ is equal to $1$ for $w=aa$ and $0$ otherwise.
The second term $\Par(\mapa_3,xwb)$ is always $0$. The third and fourth terms
$\Par(\mapa_3,xwaa)$ and $\Par(\mapa_3,xwab)$ are equal to $1$ and $0$ (respectively)
for $w=bb$, whereas they are both $0$ otherwise.
In addition, we have $\Par(\mapa_3,xba) = \Par(\mapa_3,xbb)=0$, and hence
the expression on the second line of equation~\eqref{eq:R2def} is $0$.
Adding it all up, we get $R_{3,2}(\sigma,x)=1+1\equiv 0 \pmod{2}$.

Second, suppose $x$ is not of the form $b^n$. Then for any word $w$ of two or more symbols,
we have $\Par(\mapa_3,xw)=0$, by equation~\eqref{eq:Parmapa3},
whence $R_{3,2}(\sigma,x)=0$, as desired.
\end{proof}

\begin{thm}
\label{thm:PinkSpecial}
$G^{\Pink}_{3,2,\infty}=\tilde{B}_{3,2,\infty}$.
\end{thm}

\begin{proof}
For each integer $n\geq 1$, define
\[ \tilde{B}'_{3,2,n}:= \{ \sigma\in B'_{3,2,n} : \forall m < n-4 \text{ and } \forall x\in \{a,b\}^m, \,
R_{3,2}(\sigma,x)=0\} . \]
where we recall $R_{3,2}$ takes values in $\ZZ/2\ZZ$,
and where $B'_{r,s,n}$ is as in the proof of Theorem~\ref{thm:PinkLongTail}.
By Proposition~\ref{prop:PinkSubSpec} and Definition~\ref{def:R2root}, we have
\[ G^{\Pink}_{3,2,n} \subseteq \tilde{B}_{3,2,n} \subseteq \tilde{B}'_{3,2,n} . \]
Thus, it suffices to show that $|\tilde{B}'_{3,2,n}| \leq |G^{\Pink}_{3,2,n}|$ for all $n\geq 1$.

We proceed by induction on $n$.
For $n\leq 3$, we have $\tilde{B}'_{3,2,n}= \Aut(T_n)$, and hence
\[ \log_2\big|\tilde{B}'_{3,2,n}\big| = \log_2|\Aut(T_n)| = 2^n-1= \log_2\big|G^{\Pink}_{3,2,n}\big|, \]
recalling formula~\eqref{eq:pinksize} for $\log_2 |G^{\Pink}_{r,s,n}|$.
For $n=4$, we have $\tilde{B}'_{3,2,4}=B'_{3,2,4}$,
so by the argument of Theorem~\ref{thm:PinkLongTail}, we obtain
\[ \log_2\big|\tilde{B}'_{3,2,4}\big|=\log_2\big|B'_{3,2,4}\big|\leq 13 =  \log_2\big|G^{\Pink}_{3,2,4}\big| .\]

Now let $n>4$, and suppose $|\tilde{B}'_{3,2,n-1}|\leq |G^{\Pink}_{3,2,n-1}|$. 
Let $\tilde{S}_{3,2,n}$ denote the kernel of the map
$\res_{n,n-1} : \tilde{B}'_{3,2,n}\rightarrow \tilde{B}'_{3,2,n-1}$, so that
\[ \big|\tilde{B}'_{3,2,n}\big| = \big|\res_{n,n-1}(\tilde{B}'_{3,2,n})\big|\cdot \big|\tilde{S}_{3,2,n}\big|
\leq \big|\tilde{B}'_{3,2,n-1}\big|\cdot \big|\tilde{S}_{3,2,n}\big| .\]
As before, let $Y_n$ be the kernel of $\res_{n,n-1}:\Aut(T_n)\to\Aut(T_{n-1})$.
Then by the definition of $\tilde{B}'_{3,2,n}$ in terms of $B'_{3,2,n}$,
the subgroup $\tilde{S}_{3,2,n}$ is the set of $\sigma\in Y_n$ for which
condition~\eqref{eq:condition_B_lt} holds for all nodes $x$ at level $n-4$,
and for which the further condition
\begin{equation}
\label{eq:condition_B_slt}
R_{3,2}(\sigma,y) = 0\in \ZZ/2\ZZ
\end{equation}
holds for nodes $y$ at all levels $0\leq m<n-4$.
Since $\sigma\in Y_n$ acts trivially on $T_{n-1}$, 
condition \eqref{eq:condition_B_slt} is trivially satisfied for every node $y$ at levels $m<n-5$.

Thus, $\sigma\in Y_n$ belongs to $\tilde{S}_{3,2,n}$ if and only if,
on the one hand,
condition~\eqref{eq:longtaileven} holds for all nodes $x$ at level $n-4$,
and on the other hand,
\begin{equation}\label{eq:slt2}
\sum_{w\in \{a,b\}^2}\sum_{t\in\{a,b\}}\Par(\sigma,ywat) \quad \text{is even}
\end{equation}
for all nodes $y$ at level $n-5$.

Therefore, to determine whether $\sigma\in Y_n$ belongs to $\tilde{S}_{3,2,n}$,
then for each node $y$ at level $n-5$, the values of
$\Par(\sigma,ywat)$ for $w\in \{a,b\}^{2}$ and $t\in\{a,b\}$
can be arbitrary except for $\Par(\sigma,ya^{4})$,
which must be chosen so that the sum in condition~\eqref{eq:slt2} is even.
Since there are $2^{n-5}$ nodes at level $n-5$,
there are $2^{n-5}$ parity restrictions imposed in this step.
Next, as in the proof of Theorem~\ref{thm:PinkLongTail},
for each node $x$ at level $n-4$, the values of 
$\Par(\sigma,xaw)$ for $w\in \{a,b\}^{2}$ 
that have not already been set can be arbitrary except for $\Par(\sigma,xab^{2})$, 
which must be chosen so that the first sum in condition~\eqref{eq:longtaileven} is even.
Similarly, the values of $\Par(\sigma,xbw)$ for $w\in \{a,b\}^{2}$
that have not already been set can be arbitrary except for $\Par(\sigma,xbb^{2})$, 
which must be chosen so that the second sum in condition~\eqref{eq:longtaileven} is even.
Together, these steps impose a further $2\cdot 2^{n-4}$ parity restrictions.
There are no restrictions on
the parities at the remaining  nodes at level $n-1$, and therefore
\[\log_2\big|\tilde{S}_{3,2,n}\big|= 2^{n-1}-2^{n-5} - 2^{n-3} = 2^{n-1}-5\cdot 2^{n-5}.\]

By equation~\eqref{eq:pinksize}, 
\begin{align*}
\log_2 \big|G^{\Pink}_{3,2,n}\big|-\log_2 \big|G^{\Pink}_{3,2,n-1}\big|
& = (2^n-5\cdot 2^{n-4}+2) -(2^{n-1}-5\cdot 2^{n-5}+2)\\
&=2^{n-1}-5\cdot 2^{n-5} = \log_2\big|\tilde{S}_{3,2,n}\big|.
\end{align*}
It follows that $\log_2 |G^{\Pink}_{3,2,n-1}| + \log_2|\tilde{S}_{3,2,n}| = \log_2 |G^{\Pink}_{3,2,n}|$, and hence
\[ \big|\tilde{B}'_{3,2,n}\big| \leq  \big|\tilde{B}'_{3,2,n-1}\big| \cdot \big|\tilde{S}_{3,2,n}\big|
\leq  \big|G^{\Pink}_{3,2,n-1}\big| \cdot \big|\tilde{S}_{3,2,n}\big| = \big|G^{\Pink}_{3,2,n}\big|, \]
as desired.
\end{proof}

\subsection{The non-Chebyshev short-tail case}
\label{ssec:PinkShort}

When $s=1$ and $r\geq 3$, recall from Section~\ref{sec:shorttail}
that we defined a proper subgroup $\tilde{B}_{r,1,\infty}$ of $B_{r,1,\infty}$
in Definition~\ref{def:MRroot2}. Once again, it is this smaller group that
realizes $G^{\geom}$. To see this, first observe that in this case,
we have $\mapa_1=\mapt$, and for $i\geq 2$,
equation~\eqref{eq:Parmapa2} becomes
\begin{equation}
\label{eq:Parmapa4}
\Par(\mapa_i, w) = \begin{cases}
1 & \text{ if } w = a^{i-2} w_{r-1}^n a \text{ for some } n\geq 0, \\
0 & \text{ otherwise},
\end{cases}
\end{equation}
where $w_{r-1}:=ba^{r-2}$.

\begin{prop}
\label{prop:PinkSubShort}
For any integer $r\geq 3$, Pink's subgroup $G^{\Pink}_{r,1,\infty}$ is contained in $\tilde{B}_{r,1,\infty}$.
\end{prop}

\begin{proof}
By Proposition~\ref{prop:PinkSubTail}, we already have
$G^{\Pink}_{r,1,\infty}\subseteq B_{r,1,\infty}$. It suffices to show that
$\mapa_i\in\tilde{B}_{r,1,\infty}$ for each $i\in\{1,\ldots, r\}$.

Fix a node $x$ of the tree, and fix $i\in\{1,\ldots, r\}$.
We must show that $R_{r,1}(\mapa_i,x)\equiv 0\pmod{2}$,
where $R_{r,1}$ is as in Definition~\ref{def:MRroot2}.
For $i=1$, we have $\mapa_i=\mapt$, and $\Par(\mapt,y)=0$
for all of the nodes $y$ appearing in that definition.
Thus, we may assume for the rest of the proof that $i\geq 2$.

By equation~\eqref{eq:Parmapa4}, we must have $\Par(\mapa_i,xb)=0$,
and hence the first term in the sum defining $R_{r,1}(\mapa_i,x)$ is zero.
The remaining terms are all of the form $\Par(\mapa_i,xw')$, where $w'$ is a
word of length $r$ whose second symbol is $b$.
On the other hand, according to equation~\eqref{eq:Parmapa4},
all of the nodes $y$ for which $\Par(\mapa_i,y)=1$
either are $a^{i-1}$ (which has length less than $r$)
or else end with the word $ba^{r-1}$, whose second symbol is $a$.
Thus, every term in the sum defining $R_{r,1}(\mapa_i,x)$ is zero, as desired.
\end{proof}

\begin{thm}
\label{thm:PinkShort}
For any integer $r\geq 3$, we have
$G^{\Pink}_{r,1,\infty}=\tilde{B}_{r,1,\infty}$.
\end{thm}

\begin{proof}
For each integer $n\geq 1$, define
\[ \tilde{B}'_{r,1,n}:= \{ \sigma\in B'_{r,1,n} : \forall m < n-r \text{ and } \forall x\in \{a,b\}^m,
R_{r,1}(\sigma,x)=0\} . \]
where we recall $R_{r,1}$ takes values in $\ZZ/2\ZZ$,
and where $B'_{r,s,n}$ is as in the proof of Theorem~\ref{thm:PinkLongTail}.
By Proposition~\ref{prop:PinkSubShort} and Definition~\ref{def:MRroot2}, we have
\[ G^{\Pink}_{r,1,n} \subseteq \tilde{B}_{r,1,n} \subseteq \tilde{B}'_{r,1,n} . \]
Thus, it suffices to show that $|\tilde{B}'_{r,1,n}| \leq |G^{\Pink}_{r,1,n}|$ for all $n\geq 1$.

We proceed by induction on $n$.
For $n\leq r$, we have $\tilde{B}'_{r,1,n}= \Aut(T_n)$, and hence
\[ \log_2 \big|\tilde{B}'_{r,1,n}\big| = \log_2\big|\Aut(T_n)\big| = 2^n-1= \log_2\big|G^{\Pink}_{r,1,n}\big|, \]
yet again recalling formula~\eqref{eq:pinksize} for $\log_2 |G^{\Pink}_{r,s,n}|$.

Now let $n>r$, and suppose $|\tilde{B}'_{r,1,n-1}|\leq |G^{\Pink}_{r,1,n-1}|$. 
Let $\tilde{S}_{r,1,n}$ denote the kernel of the map
$\res_{n,n-1} : \tilde{B}'_{r,1,n}\rightarrow \tilde{B}'_{r,1,n-1}$, so that
\[ \big|\tilde{B}'_{r,1,n}\big| = \big|\res_{n,n-1}(\tilde{B}'_{r,1,n})\big|\cdot \big|\tilde{S}_{r,1,n}\big|
\leq \big|\tilde{B}'_{r,1,n-1}\big|\cdot \big|\tilde{S}_{r,1,n}\big| .\]
As before, let $Y_n$ be the kernel of $\res_{n,n-1}:\Aut(T_n)\to\Aut(T_{n-1})$.
Then by the definition of $\tilde{B}'_{r,1,n}$ in terms of $B'_{r,1,n}$,
the subgroup $\tilde{S}_{r,1,n}$ is the set of $\sigma\in Y_n$ for which
condition~\eqref{eq:condition_B_lt} holds for all nodes $x$ at level $n-r-1$,
and for which the further condition
\begin{equation}
\label{eq:condition_B_st}
R_{r,1}(\sigma,x)=0\in \ZZ/2\ZZ  
\end{equation}
holds for nodes $x$ at all levels $0\leq m<n-r$.
Since $\sigma\in Y_n$ acts trivially on $T_{n-1}$, 
condition \eqref{eq:condition_B_st} is trivially satisfied for every node $y$ at levels $m<n-r-1$.

Thus, $\sigma\in Y_n$ belongs to $\tilde{S}_{r,1,n}$ if and only if
for all nodes $x$ at level $n-r-1$,
condition~\eqref{eq:longtaileven} holds, and in addition,
\begin{equation}\label{eq:st2}
\sum_{w\in \{a,b\}^{r-2}} \big( \Par(\sigma,xabw) + \Par(\sigma,xbbw) \big) \quad \text{is even}.
\end{equation}

Therefore, to determine whether $\sigma\in Y_n$ belongs to $\tilde{S}_{r,1,n}$,
then for each node $x$ at level $n-r-1$,
the values of $\Par(\sigma,xaw)$ for $w\in \{a,b\}^{r-1}$
can be arbitrary except for $\Par(\sigma,xa^{r})$,
which must be chosen so that the first sum in condition~\eqref{eq:longtaileven} is even.
Next, the values of $\Par(\sigma,xbbw)$ for $w\in \{a,b\}^{r-1}$
can be arbitrary except for $\Par(\sigma,xb^{r})$,
which must be chosen so that the sum in condition~\eqref{eq:st2} is even.
Finally, the values of $\Par(\sigma,xbaw)$ for $w\in \{a,b\}^{r-2}$ 
can be arbitrary except for $\Par(\sigma,xba^{r-1})$,
which must be chosen so that the second sum in condition~\eqref{eq:longtaileven} is even.
Thus, there are $3$ parity restrictions for each of the $2^{n-r-1}$ such nodes $x$,
all of which are independent of one another.
There are no restrictions on
the parities at the remaining  nodes at level $n-1$, and therefore
\[\log_2\big|\tilde{S}_{r,1,n}\big|= 2^{n-1}-3\cdot 2^{n-r-1}.\]

By equation~\eqref{eq:pinksize}, 
\begin{align*}
\log_2 \big|G^{\Pink}_{r,1,n}\big|-\log_2 \big|G^{\Pink}_{r,1,n-1}\big|
& = (2^n-3\cdot 2^{n-r}+2) -(2^{n-1}-3\cdot 2^{n-1-r}+2)\\
&=2^{n-1}-3\cdot 2^{n-r-1} = \log_2\big|\tilde{S}_{r,1,n}\big| .
\end{align*}
It follows that $\log_2 |G^{\Pink}_{r,1,n-1}| + \log_2|\tilde{S}_{r,1,n}| = \log_2 |G^{\Pink}_{r,1,n}|$, and hence
\[ \big|\tilde{B}'_{r,1,n}\big| \leq \big|\tilde{B}'_{r,1,n-1}\big| \cdot \big|\tilde{S}_{r,1,n}\big|
\leq \big|G^{\Pink}_{r,1,n-1}\big| \cdot \big|\tilde{S}_{r,1,n}\big| = \big|G^{\Pink}_{r,1,n}\big|, \]
as desired.
\end{proof}

\section{Obtaining the arboreal Galois groups}
\label{sec:obtain}

To simplify notation for the remainder of the paper, 
$\tilde{B}_{r,s,\infty}:=B_{r,s,\infty}$ and $\tilde{M}_{r,s,\infty}:=M_{r,s,\infty}$ when $r\geq 4$ and $s\geq 2$.
(On the other hand, when $(r,s)=(3,2)$ or when $s=1$,
we keep our existing definitions of $\tilde{B}_{r,s,\infty}$ and $\tilde{M}_{r,s,\infty}$
as in Sections~\ref{sec:Special} and \ref{sec:shorttail}, respectively.)
Thus, in all cases, our arboreal Galois groups are always contained in $\tilde{M}_{r,s,\infty}$,
and we have $\tilde{B}_{r,s,\infty} = G^{\Pink}_{r,s,\infty}$.
We also define a group $H$ and homomorphism
$\tilde{P}_{r,s}: \tilde{M}_{r,s,\infty}\to H$ to be
\begin{itemize}
\item
$H:=\ZZ/2\ZZ$ and $\tilde{P}_{r,s} := P_{r,s}$ when $r\geq 4$ and $s\geq 2$,
\item
$H:=\ZZ/2\ZZ\times \ZZ/2\ZZ$ and $\tilde{P}_{r,s} := P_{3,2}\times R_{3,2}$ when $(r,s)=(3,2)$, or
\item
$H:=\ZZ/2\ZZ$ and $\tilde{P}_{r,s} := R_{r,1}$ when $r\geq 3$ and $s=1$.
\end{itemize}

\subsection{The arithmetic Galois group over function fields}
\label{ssec:PinkGroupsArith}

Let  $k$ be an arbitrary field of characteristic not equal to $2$.
Let $K=k(t)$, let $x_0=t$, let $f(z)=z^2+c\in k[z]$ for which $0$ is preperiodic
with $r\geq 3$ and $s\geq 1$, and let $K_{\infty}$ be the resulting arboreal extension.
Let $k_{\infty}$ be the constant field extension in $K_{\infty}$, i.e.,
$k_{\infty}\subseteq K_{\infty}$ is the maximal algebraic extension of $k$ in $K_{\infty}$.
In this section we show how our groups $\tilde{M}_{r,s,\infty}$
related to Pink's group $G^{\arith}=\Gal(K_{\infty}/K)$ and to the extension $k_\infty / k$.

\begin{lemma}\label{lem:kinfty}
For any $r\geq 3$ and $s\geq 1$, we have
\[ k_\infty=\begin{cases}
	k(\zeta_4) & \text{ if } r\geq 4, s\geq 2\\
	k(\zeta_8) & \text{ if } (r,s) = (3,2) \\
	k(\sqrt{2}) & \text{ if } r \geq 3, s = 1.
\end{cases} \]
\end{lemma}

\begin{proof}
Recall the definitions of $\tilde{P}_{r,s}$ and $H$ from the start of Section~\ref{sec:obtain}, 
and the definitions of $G^{\arith}=\Gal(K_{\infty}/K)$ and $G^{\geom}$ from Section~\ref{ssec:PinkSummary}.
Define
\[ \ell:=\begin{cases}
	k(\zeta_4) & \text{ if } r\geq 4, s\geq 2\\
	k(\zeta_8) & \text{ if } r = 3, s = 2\\
	k(\sqrt{2}) & \text{ if } r \geq 3, s = 1.\\
\end{cases} \]
By Lemmas \ref{lem:iroot}, \ref{lem:root2special}, and \ref{lem:root2},
we have $\ell \subseteq k_\infty$, so we can consider the homomorphism 
\[\chi:\Gal(k_\infty/k)\longrightarrow H,\]
defined by
\begin{equation}
\label{eq:chidef}
\chi: \sigma \mapsto \begin{cases}
	P & \text{ if } \sigma(\zeta_4)=(-1)^{P}\zeta_4 \text{ and } r\geq 4, s\geq 2\\
	(P,R) & \text{ if } \sigma(\zeta_4)=(-1)^{P}\zeta_4, \,\, \sigma(\sqrt{2})=(-1)^R\sqrt{2}, \text{ and } (r,s) = (3,2) \\
	R & \text{ if } \sigma(\sqrt{2})=(-1)^R\sqrt{2} \text{ and } r \geq 3, s = 1 .
\end{cases}
\end{equation}

Applying Theorems \ref{thm:MRSiroot} and \ref{thm:PRSembed} in the case $r\geq 4, s\geq 2$,
Theorems \ref{thm:R2root} and \ref{thm:R32embed} in the case $(r,s) = (3,2)$,
and Theorems \ref{thm:MRroot2} and \ref{thm:PR1embed} in the case $r\geq 3, s = 1$,
there is an equivariant injective homomorphism $\rho: G^{\arith}\hookrightarrow \tilde{M}_{r,s,\infty}$
making the following diagram commute.
By Theorems \ref{thm:PinkLongTail}, \ref{thm:PinkSpecial}, and \ref{thm:PinkShort},
the restricted map  $\rho: G^{\geom}\to \tilde{B}_{r,s,\infty}$ is an isomorphism. 

\begin{equation}
\label{eq:maindiagram}
\begin{tikzcd}
0 \arrow[r]
  & G^{\geom} \arrow[r] \arrow[d, "\wr"]
  & G^{\arith} \arrow[r] \arrow[d, hook, "\rho"]
  & \Gal(k_\infty/k) \arrow[r] \arrow[d, "\chi"]
  & 0
\\
0 \arrow[r]
  & \tilde{B}_{r,s,\infty} \arrow[r]
  & \tilde{M}_{r,s,\infty} \arrow[r, "\tilde{P}_{r,s}"]
  & H 
\end{tikzcd}
\end{equation}

Now consider the induced homomorphism
$\bar{\rho}: G^{\arith}/G^{\geom}\to\tilde{M}_{r,s,\infty}/\tilde{B}_{r,s,\infty}$ given by 
\[\bar{\rho}(\sigma G^{\geom})=\rho(\sigma)\tilde{B}_{r,s,\infty}.\]
We claim that $\bar{\rho}$ is injective. Indeed, suppose
$\bar{\rho}(\sigma G^{\geom}) = e \tilde{B}_{r,s,\infty}$,
and hence $\rho(\sigma)\in \tilde{B}_{r,s,\infty}$.
Since $\rho$ is injective and $\rho(G^{\geom})=\tilde{B}_{r,s,\infty}$,
it follows that $\sigma\in G^{\geom}$, proving that $\bar{\rho}$ has trivial kernel, as claimed.
Consider the following commutative diagram:
\[ \begin{tikzcd}
	G^{\arith}/G^{\geom} \arrow[r, "\sim"] \arrow[d, hook, "\bar{\rho}"]
	& \Gal(k_\infty/k) \arrow[d, "\chi"]
	\\
	\tilde{M}_{r,s,\infty}/\tilde{B}_{r, s,\infty} \arrow[r]
	& H.
\end{tikzcd} \]
Since $\bar{\rho}$ and $\tilde{M}_{r,s,\infty}/\tilde{B}_{r, s,\infty}\to H$ are injective,
and $G^{\arith}/G^{\geom}\to \Gal(k_\infty/k)$ is an isomorphism, the map $\chi$ must also be injective.
Thus, $\Gal(k_\infty/\ell) = \ker\chi$ is trivial, whence $k_\infty=\ell$.
\end{proof}

\begin{lemma}
\label{lem:Ponto}
For any $r\geq 3$ and $s\geq 1$, 
the homomorphism $\tilde{P}_{r,s}:\tilde{M}_{r,s,\infty} \to H$ is surjective.
\end{lemma}

\begin{proof}
Let $c\in\Qbar$ be a root of the polynomial $F_{r,s}(x)=f_x^r(0)+f_x^s(0)\in\QQ[x]$
such that the $r+1$ iterates $\{f_c^i(0) \, | \, 0\leq i\leq r\}$ are all distinct.
Such $c$ does indeed exist;
as noted in the proof of Lemma~\ref{lem:valFrs},
$c$ is a root a Misiurewicz polynomial, 
as described in Section~1 of \cite{Gok20} or equation~(1) of \cite{BG23}.
Let $k=\QQ(c)$.

If $s\geq 2$, then by Proposition~\ref{prop:rootsnotQc}, we have $[k(\zeta_8):k]=4$;
and if $s=1$, then by Proposition~\ref{prop:root2notQc}, we have $[k(\sqrt{2}):k]=2$.
Therefore, the homomorphism $\chi$ of equation~\eqref{eq:chidef} is surjective.
In diagram~\eqref{eq:maindiagram}, then, the composed map $G^{\arith}\to H$
is surjective as well, and hence so is $\tilde{P}_{r,s}$.
\end{proof}

\begin{remark}
Lemma~\ref{lem:Ponto} can also be proven directly from the definitions of the
groups $\tilde{B}_{r,s,\infty}$ and $\tilde{M}_{r,s,\infty}$, by finding explicit elements
of $\tilde{M}_{r,s,\infty}$ mapping to each element of $H$ under $\tilde{P}_{r,s}$.
\end{remark}

\begin{thm}\label{thm:condMisGarith}
Let $K=k(t)$, let $x_0=t$, let $f(z)=z^2+c\in k[z]$ for which $0$ is preperiodic
with $r\geq 3$ and $s\geq 1$, let $K_{\infty}$ be the resulting arboreal extension,
and let $G^{\arith}=\Gal(K_\infty/K)$.
Then the following are equivalent:
\begin{enumerate}
	\item $G^{\arith}\cong\tilde{M}_{r,s,\infty}$
	\item $\begin{cases}
		[k(\zeta_4):k]=2 & \text{ if } r\geq 4, s\geq 2 \\
		[k(\zeta_8):k]=4 & \text{ if } r= 3, s=2 \\
		[k(\sqrt{2}):k]=2 & \text{ if } r\geq 3, s=1.
	\end{cases}$
\end{enumerate}
\end{thm}

\begin{proof}
By Lemma~\ref{lem:Ponto}, the diagram~\eqref{eq:maindiagram}
in the proof of Lemma~\ref{lem:kinfty} expands to the following commutative
diagram with both rows exact:
\[ \begin{tikzcd}
	0 \arrow[r]
	& G^{\geom} \arrow[r] \arrow[d, "\wr"]
	& G^{\arith} \arrow[r] \arrow[d, hook, "\rho"]
	& \Gal(k_\infty/k) \arrow[r] \arrow[d, hook, "\chi"]
	& 0
	\\
	0 \arrow[r]
	& \tilde{B}_{r,s,\infty} \arrow[r]
	& \tilde{M}_{r,s,\infty} \arrow[r, "\tilde{P}_{r,s}"]
	& H \arrow[r]
	& 0,
\end{tikzcd} \]
Therefore the map $\rho$ is an isomorphism if and only if the map $\chi$ is an isomorphism. 
\end{proof}


\subsection{Conditions over general fields}

Throughout this section, let $k$ be a field of characteristic not equal to $2$,
let $f(z)=z^2+c\in k[z]$, and let $x_0\in k$. For each $i\geq 1$, define
\begin{equation}
\label{eq:Didef}
D_i:=\begin{cases}
	x_0-c & \text{ if } i=1,\\
	f^i(0)-x_0 & \text{ if } i\geq 2.
\end{cases}
\end{equation}
It is well known (see \cite[Proposition~3.2]{AHM} and also, the proof of  \cite[Lemma~7.2]{BGJT2}) 
that $\sqrt{D_i}\in k(f^{-i}(x_0))$ for every $i\geq 1$,
since the discriminant of the polynomial $f^i(z)-x_0$ is a square in $k$ times $D_i$.

\begin{lemma}\label{lem:disc_condition}
With notation as above, let
$K_{x_0,n}:=k(f^{-n}(x_0))$ and $G_n:=\Gal(K(f^{-n}(x_0))/k)$.
For each $n\geq 1$, the following are equivalent.
\begin{enumerate}
	\item $G_n\cong \Aut(T_n)$.
	\item $[ k(\sqrt{D_1},\ldots,\sqrt{D_{n}}) : k ] =2^n$.
\end{enumerate}
\end{lemma}

\begin{proof}
This is a standard result in the study of arboreal Galois representations.
For proofs of various versions of it, see, for example,
 \cite[Lemma 4.2]{odoni_realizing},
 \cite[Lemmas 1.5 and 1.6]{Stoll},
 \cite[Section 2.2]{Jones_survey},
 \cite[Proposition 5.3]{BD}, or
 \cite[Lemma 7.2]{BGJT2}.
 \end{proof}


We can now prove our fully general version of Theorem~\ref{thm:condition}.

\begin{thm}\label{thm:mainthm}
Let $k$ be a field of characteristic not equal to $2$,
and let $f(z)=z^2+c\in k[z]$ such that $0$ is strictly preperiodic under $f$.
Let $r>s\geq 1$ be minimal so that $f^r(0)=-f^s(0)$, and assume $r\geq 3$.
Let $x_0\in k$, and define $K_{x_0,n}=k(f^{-n}(0))$, $K_{x_0,\infty}=\bigcup_{n=1}^\infty K_{x_0,n}$,
$G_{x_0,n}=\Gal(K_{x_0,n}/k)$, and $G_{x_0,\infty}=\Gal(K_{x_0,\infty}/k)$.
Further, define $D_1,\dots, D_r\in k$ as in equation~\eqref{eq:Didef},
and also define
\[\gamma:=\begin{cases}
	\sqrt{2} & \text{ if } r\geq 3 \text{ and } s=1,\\
	\zeta_8 & \text{ if } r=3 \text{ and } s=2,\\
	\zeta_4 & \text{ if } r\geq 4 \text{ and } s\geq 2,
\end{cases}
\quad\text{and}\quad
e:=\begin{cases}
 		2 & \text{ if } r=3 \text{ and } s=2, \\
 		1 & \text{ otherwise.}
\end{cases} \]
Then the following are equivalent.
\begin{enumerate}
	\item $[k(\gamma, \sqrt{D_1}, \dots, \sqrt{D_r}):k]= 2^{r+e}$.
	\item $[K_{x_0,r+e}:k]=|\tilde{M}_{r,s,r+e}|$.
	\item $G_{x_0,r+e}=\tilde{M}_{r,s,r+e}$.
	\item $G_{x_0,n}=\tilde{M}_{r,s,n}$ for all $n\geq 1$.
	\item $G_{x_0,\infty}\cong \tilde{M}_{r,s,\infty}$.
\end{enumerate}
\end{thm}

\begin{proof}
The implications (5)$\Leftrightarrow$(4)$\Rightarrow$(3)$\Leftrightarrow$(2) are trivial.
Thus, it suffices to show  (1)$\Rightarrow$(5) and (3)$\Rightarrow$(1). Define
\[ L:=k(\sqrt{D_1},\dots,\sqrt{D_r}).\]
Note that $r+e$ is the smallest integer $n$ for which
Lemmas \ref{lem:iroot}, \ref{lem:root2special}, and \ref{lem:root2} force $\gamma\in K_{x_0,n}$.
In addition, by Theorems~\ref{thm:PRSembed}, \ref{thm:R32embed}, and~\ref{thm:PR1embed},
$r+e$ is also the smallest integer $n$ for which
the quotient $\tilde{P}_{r,s,n}:\tilde{M}_{r,s,n}\to H$
(of $\tilde{P}_{r,s}$ to the bottom $n$ levels of the tree) is onto.

First suppose (1) holds, so that $[L(\gamma):k]=2^{r+e}$.
Then because $[k(\gamma):k]\leq 2^e$ and $[L(\gamma):k(\gamma)]\leq 2^r$,
we must have both $[k(\gamma):k]= 2^e$ and $[L(\gamma):k(\gamma)]= 2^r$.
Now since $[k(\gamma):k]= 2^e$, Theorem~\ref{thm:condMisGarith} tells us that
$\tilde{M}_{r,s,\infty}\cong G^{\arith}$, where
$G^{\arith}=\Gal(k(t)_{\infty}/k(t))$.

On the other hand, by Lemma~\ref{lem:disc_condition}, we have 
\[ \Gal\big( k(\gamma) \cdot K_{x_0,r} /k(\gamma)\big) \cong \Aut(T_r)\cong \tilde{M}_{r,s,r},\]
and therefore,
\begin{equation}
\label{eq:OldBGJTthm}
\big| \Gal\big( k(\gamma) \cdot K_{x_0,r} /k \big) \big|
= \big| \Gal\big( k(\gamma) \cdot K_{x_0,r} /k(\gamma)\big) \big| \cdot [k(\gamma):k]
= \big|\tilde{M}_{r,s,r}\big|\cdot [k(\gamma):k].
\end{equation}

Recall that the (strict) forward orbit $\{f^i(0) : i\geq 1\}$ of $0$ has cardinality $r$.
In addition, since $k_\infty = k(\gamma)$ by Lemma~\ref{lem:kinfty},
the compositum of all degree~$2$ extensions of $k$ contained in $k_{\infty}$
is $k_{\infty}=k(\gamma)$ itself.
These observations, including equation~\eqref{eq:OldBGJTthm}, provide precisely the
hypotheses of \cite[Theorem 4.6]{BGJT1}, and hence
$G_{x_0,\infty}\cong G^{\arith}\cong \tilde{M}_{r,s,\infty}$, proving statement (5).

Finally, suppose statement (3) holds. Then restricting to level $r$, we have
\[ G_{x_0,r} \cong \tilde{M}_{r,s,r} = \tilde{B}_{r,s,r} = G^{\Pink}_{r,s,r} = \Aut(T_r) . \]
(The last equality here holds by equation~\eqref{eq:pinksize}, and the others hold because
$\tilde{B}_{r,s,r} \subseteq \tilde{M}_{r,s,r} \subseteq  \Aut(T_r)$.)
Thus, Lemma~\ref{lem:disc_condition} implies that $[L:k]=2^r$.

As noted at the start of this proof, we have $\gamma\in K_{x_0,r+e}$,
and $\tilde{P}_{r,s,r+e}:\tilde{M}_{r,s,r+e}\to H$ is onto, so that
$| \tilde{M}_{r,s,r+e}| /  |\tilde{B}_{r,s,r+e}| = |H|=2^e$.
By condition~(3), then, we have
\begin{multline*}
\big| \tilde{B}_{r,s,r+e} \big| \cdot [k(\gamma) : k]
= [K_{x_0,r+e} : k(\gamma)] \cdot [k(\gamma) : k] 
\\
= [K_{x_0,r+e} : k] = \big| G_{x_0,r+e} \big| = \big| \tilde{M}_{r,s,r+e} \big| 
= 2^e \big| \tilde{B}_{r,s,r+e} \big|,
\end{multline*}
and hence $[k(\gamma):k]=2^e$.

In addition, since $L\subseteq K_{x_0,r}$ and $G_{x_0,r}\cong \tilde{B}_{r,s,r}$,
we must have $k(\gamma)\cap K_{x_0,r}=k$, and therefore $k(\gamma)\cap L = k$. Thus,
$[L(\gamma):L] = [k(\gamma):k]$, and hence
\[ [L(\gamma):k] = [L(\gamma):L] \cdot [L:k]
= [k(\gamma):k] \cdot [L:k]= 2^{e} \cdot 2^r = 2^{r+e} ,\]
proving statement (1).
\end{proof}

\begin{remark}
We quoted \cite[Theorem 4.6]{BGJT1} in the proof of Theorem~\ref{thm:mainthm} above.
That result is stated under the assumption that $k$ is a number field,
but as noted in \cite{BGJT1} just before its statement, this assumption is needed only so that
$\ch k\neq 2$, and so that there are only finitely many extensions of $k$ of degree $2$ in $k_{\infty}$;
but we know both of these facts already in the proof of Theorem~\ref{thm:mainthm}.

There is also a typographical error in \cite{BGJT1} when defining the length of the forward orbit of $0$.
This quantity is defined in the statement of~\cite[Lemma~4.1]{BGJT1} as $|\{f^i(0) : i\geq 0\}|$, but that
should have been the \emph{strict} forward orbit, with $i\geq 1$. Indeed, the proof of \cite[Lemma~4.1]{BGJT1}
quotes \cite[Theorem~3.1]{JKMT}, which in turn uses the strict forward orbit.
The reason that the strict forward orbit is important is that for $K=k(t)$ with $x_0=t$,
the finite primes that ramify in $K_{\infty}/K$ are precisely those of the form $(t-f^i(0))$ for $i\geq 1$.
\end{remark}

\textbf{Acknowledgments}.
The first author gratefully acknowledges the support of NSF grant DMS-2401172.

\bibliographystyle{amsalpha}
\bibliography{biblio}

\end{document}